\documentclass[10pt]{article}
\usepackage[
 includeheadfoot,
  top=2cm,
  bottom=2cm,
  left=3cm,
  right=3cm
]{geometry}

\newcommand{\papertitle}{Stable Matrix-Free Q-less QR via Preconditioning and Randomization}
\newcommand{\paperauthorone}{Oleg Balabanov}
\newcommand{\paperauthortwo}{Maksim Melnichenko}
\newcommand{\paperauthorthree}{Michael W. Mahoney}
\newcommand{\affiliationUCB}{University of California, Berkeley (UCB)}
\newcommand{\affiliationICSI}{International Computer Science Institute (ICSI)}
\newcommand{\affiliationLBNL}{Lawrence Berkeley National Laboratory (LBNL)}

\usepackage{xcolor}
\usepackage{booktabs}
\usepackage{multirow}
\usepackage{array}
\usepackage{float}
\usepackage{algorithm}
\usepackage{algpseudocode}
\usepackage{setspace}
\usepackage{amsmath, amssymb}
\usepackage{amsthm}
\usepackage[colorlinks=true,linkcolor=blue,citecolor=magenta]{hyperref}
\usepackage[capitalize]{cleveref}
\usepackage{bm} 
\providecommand{\mathbold}[1]{\bm{\mathsf{#1}}}
\usepackage{stmaryrd}
\usepackage{graphicx}
\usepackage{adjustbox}
\usepackage{placeins}
\SetSymbolFont{stmry}{bold}{U}{stmry}{m}{n}

\makeatletter
\newcommand{\footremember}[2]{%
    \begingroup
    \scriptspace=0pt\relax
    \def\@makefnmark{\hbox{\@textsuperscript{\normalfont\@thefnmark}}}%
    \footnote{#2}%
    \endgroup
    \newcounter{#1}%
    \setcounter{#1}{\value{footnote}}%
}
\newcommand{\footrecall}[1]{%
    \begingroup
    \scriptspace=0pt\relax
    \def\@makefnmark{\hbox{\@textsuperscript{\normalfont\@thefnmark}}}%
    \footnotemark[\value{#1}]%
    \endgroup
}
\makeatother

\newcommand{\code}[1]{\texttt{#1}}

\newcommand{\R}{\mathbb{R}}

\newcommand{\vct}[1]{\bm{#1}}

\newcommand{\mtx}[1]{\mathbold{#1}}

\DeclareMathOperator{\range}{range}

\DeclareMathOperator{\diag}{diag}

\DeclareMathOperator{\cond}{\mathrm{cond}}
\DeclareMathOperator{\trace}{tr}

\newcommand{\trans}{\text{\tiny{T}}}

\theoremstyle{plain}
\newtheorem{theorem}{Theorem}[section]
\newtheorem{proposition}[theorem]{Proposition}
\newtheorem{lemma}[theorem]{Lemma}

\newtheorem{corollary}[theorem]{Corollary}
\theoremstyle{definition}
\newtheorem{definition}[theorem]{Definition}
\newtheorem{assumption}[theorem]{Assumption}
\newtheorem{example}{Example}
\theoremstyle{remark}
\newtheorem{remark}{Remark}

\usepackage[utf8]{inputenc}
\usepackage[T1]{fontenc}

\hypersetup{
  pdftitle={\papertitle},
  pdfauthor={\paperauthorone, \paperauthortwo, \paperauthorthree}
}

\title{\papertitle}
\date{}

\author{
  \paperauthorone%
  \footremember{UCB}{\affiliationUCB}%
  \footremember{ICSI}{\affiliationICSI}%
  \footremember{LBNL}{\affiliationLBNL}%
  \and
  \paperauthortwo%
  \footrecall{UCB}%
  \footrecall{ICSI}%
  \and
  \paperauthorthree%
  \footrecall{UCB}%
  \footrecall{ICSI}%
  \footrecall{LBNL}%
}

\begin{document}
\maketitle
\begin{abstract}
This article develops matrix-free Q-less QR algorithms for orthogonalizing tall linear operators. Our main algorithm, CQRRTO, uses randomized preconditioning to make a single Cholesky QR pass numerically safe for any numerically full-rank operator, overcoming the stability issues of Cholesky QR and related methods that operate on the Gram matrix. Compared with deterministic stabilized variants such as Q-less Cholesky QR2 (reorthogonalized Cholesky QR), it remains numerically stable under substantially weaker conditioning assumptions while requiring fewer operator applications. Matrix-free Q-less QR can be useful for improving the numerical robustness of least-squares solves via the normal equations, spectral approximations in a low-dimensional basis, and inversion of operators with low-rank updates via the Woodbury identity.

We present a componentwise stability analysis of preconditioned Q-less Cholesky QR that applies to matrix-free operators and is modular with respect to the preconditioner, yielding guarantees for Q-less Cholesky QR, Cholesky QR2, shifted Cholesky QR3, and CQRRTO as special cases. All methods are implemented in the open-source RandLAPACK library. 
\end{abstract}

\section{Introduction}
Efficient, numerically robust algorithms for the Q-less QR factorization of a tall matrix $\mtx M\in\mathbb{R}^{m\times n}$ with $m\gg n$ can improve the stability of various computations involving the Gram~matrix
$$
\mtx G=\mtx M^{\trans}\mtx M\in\mathbb{R}^{n\times n}.
$$
Examples include solving least-squares problems via the normal equations, computing orthogonal projections onto $\range(\mtx M)$, computing Rayleigh–Ritz eigenpair approximations, and inverting operators with low-rank updates via the Woodbury identity. We discuss these examples in detail in \cref{subsec:apps}.

This paper is concerned with such algorithms in the matrix-free setting, where $\mtx M$ is available only as a \textit{linear operator}, accessed through its forward and adjoint applications $\mtx M[\cdot]$ and $\mtx M^\trans[\cdot]$. The operator setting has two important consequences. 
First, applying the operator may cost much less than multiplying by a dense matrix of the same dimensions.
Consequently, materializing and operating with $m\times n$ dense arrays may be undesired or prohibited. Second, the numerical stability guarantees must account for the implementation specific errors in the forward and adjoint calls.

Computing and then factorizing $\mtx G$ is attractive because it requires no
entrywise access to $\mtx M$ and maps well to parallel hardware~\cite{FukayaEtAl2014,YamazakiEtAl2015,TomasQuintanaOrti2019}. However, this approach entails a fundamental numerical weakness. Forming the Gram matrix squares the condition number of $\mtx M$: $\cond(\mtx G)=\cond(\mtx M)^2$. Errors in a factorization or solve can therefore grow with $\cond(\mtx M)^2$. When $\cond(\mtx M)$ approaches the $u^{-1/2}$ scale, where $u$ denotes the unit roundoff, rounding in the computed Gram matrix can destroy positive definiteness and cause the Cholesky factorization to~fail. 

A more robust alternative is to obtain the triangular factor through a QR factorization $\mtx M=\mtx Q\mtx R$.
In exact arithmetic,
\[
\mtx G =\mtx R^{\trans}\mtx R,
\]
so $\mtx R$ is precisely the Cholesky factor of the Gram matrix. A numerically stable QR algorithm can compute this factor without incurring the squared-conditioning sensitivity associated with explicitly forming $\mtx M^{\trans}\mtx M$. In the operator setting, however, the large factor $\mtx Q$ should neither be formed nor stored. Its action can instead be recovered implicitly from $\mtx M$ and $\mtx R$. For example, one can evaluate
$\mtx Q^{\trans}\vct b$ as $\mtx R^{-\trans}\mtx M^{\trans}\vct b$.
This leads naturally to the notion of a \textit{Q-less QR factorization}~\cite{Stewart1999}, in which only the triangular factor $\mtx R$ is retained. When $m\gg n$, this replaces the $\mathcal O(mn)$ storage and application costs associated with an explicit $\mtx Q$ by operator applications and $\mathcal O(n^2)$ dense linear algebra operations.

The difficulty is that the classical QR algorithms whose stability motivates this substitution are poorly suited to the Q-less operator setting. Householder and Givens QR~\cite{GolubVanLoan2013} retain transformations built from length-$m$ vectors. Gram–Schmidt explicitly constructs and stores the tall orthogonal basis. Cholesky QR (hereafter CholQR) is the exception.
It accesses $\mtx M$ only through forward and adjoint actions and produces $\mtx R$ without ever storing a tall orthogonal factor. 
Unfortunately, it computes $\mtx R$ through the explicitly formed Gram matrix, and therefore it inherits precisely the $\cond(\mtx M)^2$ scaling of rounding errors.

The goal of this work is to improve the 
numerical robustness of Q-less CholQR while retaining its effectiveness in operator setting. Our approach is based on \textit{preconditioning}. Rather than applying CholQR directly to $\mtx M$, we apply it to the composite operator
$$
\mtx M^{\rm pre}[\cdot]
=
\mtx M [\mtx P^{\rm inv}[\cdot] ],
$$
where the triangular  operator $\mtx P^{\rm inv}[\cdot]$ is chosen so that $\mtx M^{\rm pre}$ is well-conditioned. This framework allows us to develop Q-less counterparts of established stabilization techniques for CholQR~\cite{FukayaEtAl2014,YamamotoNakatsukasa2015,FukayaEtAl2020,Balabanov2022,FanGuoLin2021,MelnichenkoEtAl2024}.

The main algorithm introduced in this work is CQRRTO (Cholesky QR with Randomization for Tall Operators), a randomized preconditioned Q-less CholQR method designed specifically for the operator setting. CQRRTO computes $\mtx P^{\rm inv}[\cdot]$ from a small randomized sketch of $\mtx M$ providing a well-conditioned $\mtx M^{\rm pre}[\cdot]$ with a high probability. It achieves a linear loss of orthogonality scaling $\mathcal O(u\,\cond(\mtx M))$ throughout the entire numerically meaningful regime $\cond(\mtx M)\lesssim u^{-1}$. At the same time, it requires fewer operator applications than deterministic Q-less stabilization strategies. CQRRTO uses $3n$ operator matvecs, 
compared with $4n$ for Q-less Cholesky QR2 and $6n$ for Q-less shifted Cholesky QR3 (see \cref{tab:stability-summary}).

Each algorithm presented in this work is accompanied by a rigorous finite-precision stability analysis that describes errors in the forward and adjoint operator applications through error-shaping matrices and propagates those errors through the preconditioned algorithms. The resulting guarantees therefore apply to a broad class of implicitly defined operators. 

More broadly, this work contributes to the growing body of research on Randomized Numerical Linear Algebra (RandNLA) by demonstrating a role for randomization beyond its classical use for reducing computational complexity~\cite{Mah-mat-rev_BOOK,DM16_CACM,MurrayEtAl2023}. In particular, randomized sketching can serve as a mechanism for \textit{numerical stabilization}, making computations robust in conditioning regimes where their deterministic counterparts become numerically unreliable.

\subsection{Definitions and notation}
\label{subsec:def_not}
\paragraph{General notation.}
Matrices appear in boldface capital letters and vectors in boldface lowercase letters.
The transpose of a matrix $\mtx{X}$ is $\mtx{X}^{\trans}$, and its elementwise absolute value is $|\mtx{X}|$.
We write $\sigma_{\min}(\mtx{X})$ for the smallest singular value of $\mtx{X}$, $\cond(\mtx{X})$ for its $\ell_2$ condition number, and $\cond_F(\mtx{X}) = \|\mtx{X}\|_F/\sigma_{\min}(\mtx{X})$ for its Frobenius condition number.
The identity matrix is denoted by $\mtx{I}$.

We use $u$ to denote the unit roundoff in floating-point error models and $\mathrm{fl}(\cdot)$ to denote floating-point evaluation; computed quantities carry a hat, as in $\widehat{\mtx{X}}$.

\paragraph{Operator notation.}
Abstract linear operators and their applications are denoted with square brackets: for a linear operator $\mtx{M}$, we write $\mtx{M}[\cdot]$ for the routine that applies it and $\mtx{M}^{\trans}[\cdot]$ for the routine that applies its adjoint, so that $\mtx{M}[\mtx{X}]$ is the computed, finite-precision counterpart of the exact product $\mtx{M}\mtx{X}$. Composite applications are bracketed accordingly, as in $\mtx{M}^{\trans}[\mtx{M}[\mtx{I}]]$. \Cref{sec:impl} describes how the bracketed applications are realized in software.

\subsection{Applications}
\label{subsec:apps}
 
The Q-less QR is useful whenever the downstream computation needs a transform that orthogonalizes $\mtx M$. We give several representative examples.

\paragraph{Tall least-squares problems for fast implicit or generated-on-the-fly operators.}
A first application is the solution of poorly conditioned tall least-squares problems,
\begin{equation}\label{eq:lsq}
\min_{\vct{x}}
\left\|
\mtx M \vct{x}
-
\vct b
\right\|_2,
\end{equation}
where $\mtx M $ and $\mtx M^{\trans}$ can be efficiently applied to vectors and matrices in terms of flops and/or storage. 
Such scenarios arise in nonlinear least-squares and inverse problems, where $\mtx M$ is a Jacobian operator~\cite{NocedalWright2006,Bjorck1996}; in minimum-residual Petrov--Galerkin discretizations of PDEs, where $\mtx M  = \mtx L^{-1}\mtx K \mtx U$ with $\mtx U$ a sparse trial basis and $\mtx L\mtx L^{\trans}$ the test-space Riesz map~\cite{DemkowiczGopalakrishnan2011} 
in random-feature and kernel regression, where rows of $\mtx M$ can be efficiently regenerated~\cite{RahimiRecht2007,AvronClarksonWoodruff2017,ChenChiEYang2022}; and in tall Toeplitz least-squares from deconvolution, autoregression, and system identification, where $\mtx M $ is applied in $\mathcal{O}(m\log m)$ flops by FFT~\cite{ChanNagyPlemmons1993,Ljung1999,ChanNg1996,DeRoeck2002,ChenChenAndersen2026}. 
A Q-less QR factorization that retains only the triangular factor
$\mtx R$ is naturally suited to these cases. The solution then follows from solving the preconditioned system $\mtx R^{-\trans} \mtx M^{\trans} \mtx M \mtx R^{-1} \vct y = \mtx R^{-\trans} \mtx M^{\trans}\vct b$ with a few refinement steps guaranteeing backward 
stability~\cite{RozloznikSmoktunowiczKopal2014,Bjorck1987,EpperlyMeierNakatsukasa2024}. Unlike sketch-and-precondition solvers for RandNLA~\cite{RokhlinTygert2008,Blendenpik2010,LSRN2014}, which consider sketching for constructing coarse preconditioners, CQRRTO returns an accurate, reusable preconditioning factor with $\mtx R^{-\trans} \mtx M^{\trans} \mtx M \mtx R^{-1} - \mtx I$ on the order of $u\cond(\mtx M)$, enabling convergence in a few
iterations. It is most competitive when the sequential operator applications in iterative solvers entail a high communication cost or when the least-squares problem has many right-hand sides.

\paragraph{Reduced spectral approximation of operators.}
A second application is reduced-basis spectral approximation. Given an operator $\mtx C \in \mathbb{R}^{m\times m}$ and an implicit basis
$\mtx M \in \mathbb{R}^{m\times n}$ that approximately spans an eigenspace
of $\mtx C$, the standard Rayleigh--Ritz approach orthonormalizes
$\mtx M$ into $\mtx Q$ and solves
\[
\mtx H\vct u=\theta\vct u,
\qquad
\mtx H=\mtx Q^{\trans}\mtx C\mtx Q,
\]
yielding the Ritz pairs $(\theta,\mtx Q\vct u)$.
Our Q-less QR factorization allows this projection
without forming orthonormal factor $\mtx Q$. Instead, we compute only the factor $\mtx R$, form the reduced matrix
\[
\mtx H=\mtx R^{-\trans}\mtx M^{\trans}\mtx C\mtx M\mtx R^{-1},
\]
and represent the resulting Ritz pairs implicitly as
$(\theta,\mtx M\mtx R^{-1}\vct u)$.

For example, consider inverse subspace iteration for
$\mtx A\vct x=\lambda\mtx B\vct x$, with 
$\mtx B=\mtx L\mtx L^{\trans}$.
To target eigenvalues near a shift $\mu$, we can use
$\mtx M=\mtx C^{-j}\mtx U$, where
$\mtx C^{-1}=\mtx L^{\trans}(\mtx A-\mu\mtx B)^{-1}\mtx L$
and $\mtx U$ is a sparse starting basis, such as a coarse-mesh
space~\cite{XuZhou2001,XuXieZhang2020} or a random sparse matrix. Explicitly forming and orthogonalizing $\mtx M$
can incur substantial storage and computational costs.
Moreover, repeated application of $\mtx C^{-1}$ can make
$\mtx M$ poorly conditioned, motivating our robust,
matrix-free Q-less QR.

\paragraph{Computing orthogonal projections.}
Another direct application is the computation of orthogonal projections
onto an implicit basis $\mtx M$. Given a Q-less
QR factorization $\mtx M=\mtx Q\mtx R$, the orthogonal projector onto
$\range(\mtx M)$ can be applied implicitly via
\[
\mtx\Pi
=
\mtx M\mtx R^{-1}\mtx R^{-\trans}\mtx M^{\trans}
\]
without forming an orthonormal basis.
Its numerical accuracy depends critically
on the
numerical stability of the Q-less QR method.

This construction was used in RandRAND, our recent work on randomized preconditioning for solving large linear systems~\cite{balabanov2025preconditioning}, where $\mtx M$ is taken as an implicit randomized basis.
More broadly, our stabilization mechanisms may be useful for
robustifying RandNLA algorithms that apply low-dimensional randomized
projections implicitly~\cite{woodruff2014sketching}. Such constructions
play a central role in recent RandNLA algorithms; see, e.g.,~\cite{derezinski2025faster,derezinski2026approaching}.

\paragraph{Inversion of implicit normal updates.}
A fourth application is the stable inversion or factorization of a regularized low-rank update
\[
 \mtx \Sigma = \bigl( \mtx B + \mtx V\mtx V^{\trans}\bigr)^{-1},
 \qquad
 \mtx B = \mtx L\mtx L^{\trans},
\]
with $\mtx L$ sparse and cheaply invertible and $\mtx V$ a tall implicit operator. Such updates arise in large-scale Bayesian inverse problems and uncertainty quantification~\cite{iic_TR}, where $\mtx B$ is the prior precision and $\mtx V\mtx V^{\trans}$ the data-misfit Hessian, so that $\mtx \Sigma$ is the posterior covariance~\cite{BuiThanhEtAl2013,Spantini2015}. The Hessian is typically implicit, accessed only through its action, for example, through PDE solves. The same algebraic structure appears in limited-memory second-order optimization, where $\mtx V$ is a curvature update~\cite{ByrdNocedalSchnabel1994,NocedalWright2006}. One needs either to apply $\mtx \Sigma$ or to obtain an implicit factor $\mtx D$ with $\mtx \Sigma=\mtx D^{\trans}\mtx D$, for instance to draw samples $\mtx D^{\trans}\vct\xi$.

By the Woodbury identity, $\mtx \Sigma = \mtx B^{-1} - \mtx B^{-1}\mtx V(\mtx I+\mtx V^{\trans}\mtx B^{-1}\mtx V)^{-1}\mtx V^{\trans}\mtx B^{-1}$, so the task reduces to inverting the capacitance matrix $\mtx I+\mtx V^{\trans}\mtx B^{-1}\mtx V$, whose explicit formation can lead to severe rounding errors proportional to the conditioning of the matrix.
Numerical stability is a well-recognized challenge in computations based on the Woodbury formula; see, e.g., the recent works~\cite{ma2026note,hashemi2025instability}.
Writing $\mtx Y=\mtx L^{-1}\mtx V$, so that $\mtx V^{\trans}\mtx B^{-1}\mtx V = \mtx Y^{\trans}\mtx Y$, the numerical stability can be improved by a numerically robust QR of the augmented operator
\[
{\mtx M}
=
\begin{bmatrix} \mtx Y\\ \mtx I \end{bmatrix},
\qquad
\mtx M^{\trans}\mtx M
=
\mtx I+\mtx Y^{\trans}\mtx Y
=
\mtx R^{\trans}\mtx R.
\]
Since $\mtx Y$ is available only through its action, $\mtx R$ should be computed by a Q-less QR, forming none of $\mtx V$, $\mtx Y$, $\mtx M$, or $\mtx Q$. The inverse capacitance is then $\mtx R^{-1}\mtx R^{-\trans}$, applied by two triangular solves, giving the stabilized matrix-free representation $\mtx \Sigma = \mtx B^{-1} - \mtx B^{-1}\mtx V\,\mtx R^{-1}\mtx R^{-\trans}\,\mtx V^{\trans}\mtx B^{-1}$. 

The factor $\mtx R$ also enables stable matrix-free application of the orthogonal projector onto the complement of $\mathrm{range}(\mtx M)$. A straightforward algebraic construction using this projector yields an implicit square-root factor $\mtx D$ satisfying $\mtx\Sigma=\mtx D^{\trans}\mtx D$.

The same Q-less principle can be exploited in broader contexts, for instance to resolve redundant linear constraints in large models~\cite{ChoEtAl2025}.

\subsection{Summary of contributions}

\paragraph{Prior work.}
Our methods draw on two lines of work: the randomized preconditioning of CholQR, and the stabilization of CholQR by reorthogonalization.

\textit{Randomized preconditioning.}
Randomized (sketched) orthogonalization was introduced by Balabanov and Grigori, first for the Gram--Schmidt process~\cite{BalabanovGrigori2022} and later in a block form~\cite{BalabanovGrigori2025}. The idea of sketching $\mtx{M}$ with a subspace embedding to form a smaller matrix $\mtx{M}^{\mathrm{sk}} = \mtx{S}\mtx{M}$ whose triangular factor $\mtx{R}^{\mathrm{sk}}$ serves as an approximate orthogonalization transform for $\mtx{M}$  was noted in~\cite[Remark~2.10]{BalabanovGrigori2022}, and then used as an explicit subroutine (termed RCholQR) within the randomized block Gram--Schmidt algorithm~\cite{BalabanovGrigori2025}. RCholQR was analyzed in finite precision and considered for preconditioning CholQR by Balabanov in~\cite{Balabanov2022}. This randomized preconditioning principle forms the foundation of the CQRRTO algorithm introduced in this paper. Melnichenko et al.~\cite{MelnichenkoEtAl2024} extended it to column-pivoted QR (CQRRPT), where the sketch is also used to determine the pivot ordering allowing CQRRPT to attain unconditional numerical stability. Fan et al.~\cite{FanGuoLin2021} independently proposed a randomized preconditioned CholQR in a similar spirit, although without the subspace-embedding perspective. Higgins et al.~\cite{HigginsEtAl2025} analyzed a randomized Householder--Cholesky QR, providing a detailed stability analysis valid for embedding distortions across nearly the full admissible range. Garrison and Ipsen~\cite{GarrisonIpsen2024} developed a modular perturbation analysis of the method. All of these algorithms form the orthonormal factor $\mtx{Q}$ explicitly, operate on an explicitly stored matrix, and are analyzed in the normwise sense.

\textit{Reorthogonalization.}
The classical stabilization of CholQR by reorthogonalization can also be viewed as preconditioning. Cholesky QR2 (hereafter CholQR2)~\cite{FukayaEtAl2014,YamamotoNakatsukasa2015} restores $\mathcal{O}(u)$ orthogonality whenever $\cond(\mtx{M}) \lesssim u^{-1/2}$, and shifted Cholesky QR3 (hereafter sCholQR3)~\cite{FukayaEtAl2020} extends the reach to $\cond(\mtx{M}) \lesssim u^{-1}$, explicitly casting its first (shifted) pass as a preconditioning step. Both methods form the intermediate orthonormal factors explicitly. Terao et al.~\cite{TeraoOzakiOgita2020} use related preconditioning with LU factors. Orthogonalization has also been studied in nonstandard (oblique) inner products defined by a weight matrix~\cite{LoweryLangou2014,YamamotoEtAl2016oblique}. When this weight is available in factored form, such an oblique orthogonalization reduces to an $\ell_2$-QR of an implicitly defined operator and can therefore be performed and stabilized by the matrix-free, Q-less methods developed here. The existing analyses instead form the orthonormal factor and treat the weight as an explicitly applied matrix rather than an implicit linear operator.

A related Q-less construction of orthogonal projections onto an implicit random basis appeared in our recent work on randomized preconditioners~\cite{balabanov2025preconditioning}. In particular, Algorithm~3 therein can be viewed as preconditioned Q-less QR. The present work develops this construction into a general method with rigorous guarantees of numerical~stability.

\paragraph{Contributions.}
We summarize the main contributions as follows.

\begin{itemize}
\item \textbf{A fast, robust matrix-free algorithm.}
The central contribution of this work is CQRRTO, an algorithm that uses sketching-based preconditioning to make a single CholQR pass numerically safe for any full-rank operator with $\cond(\mtx{M}) \lesssim u^{-1}$. We develop CQRRTO within a general preconditioning framework for Q-less CholQR that includes the Q-less CholQR2 and sCholQR3 baselines as special cases. CQRRTO extends the range of applicability of both baselines while requiring fewer operator applications. To our knowledge, it is the first robust matrix-free, Q-less orthogonalization algorithm.

\item \textbf{Modular stability analysis.}
We give a numerical stability analysis of preconditioned Q-less CholQR under a general model that, via error-shaping matrices, covers matrix-free operators. The analysis is modular in the choice of preconditioner and yields, as special cases: Q-less CholQR, with a loss of orthogonality $\mathcal{O}(u\,\cond(\mtx{M})^2)$; Q-less CholQR2 and Q-less sCholQR3, which improve this to $\mathcal{O}(u\,\cond(\mtx{M}))$ but require $\cond(\mtx{M}) \lesssim u^{-1/2}$ and $\cond(\mtx{M}) \lesssim u^{-3/4}$, respectively; and CQRRTO, which attains the same $\mathcal{O}(u\,\cond(\mtx{M}))$ while extending the safe range to $\cond(\mtx{M}) \lesssim u^{-1}$.
{To our knowledge, no prior work analyzed the stability of these algorithms in Q-less or matrix-free settings.}

\item \textbf{Implementation and evaluation.}
The Q-less methods are implemented in the open-source RandLAPACK
library~\cite{randlapack} on top of its linear-operator abstraction, so that
CQRRTO applies to any operator that provides forward and adjoint applications.
We validate their numerical stability on synthetic operators and evaluate their performance as preconditioners for least-squares problems on two matrix-free benchmarks. We compare CQRRTO with Q-less variants of CholQR, CholQR2, and sCholQR3, a sketch-and-precondition baseline, and the unpreconditioned solver in terms of orthogonality loss, iteration counts, solution accuracy, and wall-clock time.

\end{itemize}

The numerical guarantees and computational costs 
are summarized in
\cref{tab:stability-summary}. 
Among the stabilized Q-less methods, CQRRTO provides both the widest admissible conditioning range and the lowest computational cost.

\begin{table}[htb]
\caption{
Summary of the Q-less CholQR variants proposed and analyzed in this
paper, for column block size $b$. Here, $T_{\rm mv}$ denotes the cost of
one forward or adjoint operator application to a vector. For CQRRTO, we
assume a sketch dimension $d=\mathcal O(n)$ and a sparse embedding with
$\zeta$ nonzeros per column; typically, $\zeta=4$ or $8$, so that
$\zeta mn\ll nT_{\rm mv}$. The peak-memory bounds assume that tall
intermediates are processed in blocks and never stored explicitly.
The $\mathcal O(\cdot)$ notation in the
stability bounds suppresses factors polynomial in $n$. We here assume
that the normwise perturbations in forward and adjoint applications
to vectors $\vct x$ are bounded by
$\mathcal{O}(u)\,\|\mtx M\|_F\|\vct x\|_2$.
The admissible range denotes the
conditioning regime in which the corresponding sufficient stability
conditions can be~satisfied.
}
\label{tab:stability-summary}
\centering
\footnotesize
\setlength{\tabcolsep}{3pt}

\def\summarytablewidth{\textwidth}

\begin{adjustbox}{max width=\summarytablewidth}
\begin{tabular}{@{}llllll@{}}
\toprule
\shortstack[c]{Q-less QR method}
& \shortstack[c]{Operator\\applications}
& \shortstack[c]{Peak\\memory}
& \shortstack[c]{Complexity}
& $\|\mtx Q^{\trans}\mtx Q-\mtx I\|_F$
& \shortstack[c]{Admissible\\range}
\\
\midrule
CholQR
& $2\lceil n/b\rceil$
& $\mathcal O(mb+n^2)$
& $2nT_{\rm mv}+\mathcal O(n^3)$
& $\mathcal O(u\,\cond(\mtx M)^2)$
& $\cond(\mtx M)\lesssim u^{-1/2}$
\\
\shortstack[l]{CholQR2 (our work)}
& $4\lceil n/b\rceil$
& $\mathcal O(mb+n^2)$
& $4nT_{\rm mv}+\mathcal O(n^3)$
& $\mathcal O(u\,\cond(\mtx M))$
& $\cond(\mtx M)\lesssim u^{-1/2}$
\\
\shortstack[l]{sCholQR3 (our work)}
& $6\lceil n/b\rceil$
& $\mathcal O(mb+n^2)$
& $6nT_{\rm mv}+\mathcal O(n^3)$
& $\mathcal O(u\,\cond(\mtx M))$
& $\cond(\mtx M)\lesssim u^{-3/4}$
\\
\shortstack[l]{CQRRTO (our work)}
& $3\lceil n/b\rceil$
& $\mathcal O(m(b+\zeta)+n^2)$
& $3nT_{\rm mv}+\zeta mn+\mathcal O(n^3)$
& $\mathcal O(u\,\cond(\mtx M))$
& $\cond(\mtx M)\lesssim u^{-1}$
\\
\bottomrule
\end{tabular}
\end{adjustbox}
\end{table}

\subsection{Outline}

\Cref{sec:precond_cholqr} introduces the preconditioning framework for Q-less CholQR and its CholQR2 and sCholQR3 variants.
\Cref{sec:precond_via_sk} develops random sketching-based preconditioning, yielding CQRRTO.
\Cref{sec:theory} analyzes rounding errors in preconditioned operator applications and derives stability bounds for ordinary, preconditioned, and randomized Q-less CholQR.
\Cref{sec:experiments} evaluates stability and downstream accuracy on synthetic problems and least-squares performance on matrix-free finite-element (FEM) and Toeplitz benchmarks.
\Cref{sec:conclusion} gives concluding remarks.
Appendix~\ref{sec:impl} describes the RandLAPACK implementations and compared methods;
Appendix~\ref{sec:ls_set} documents the benchmark configuration and additional results;
and Appendix~\ref{app:proofs} provides the full theoretical results and proofs.

\section{Preconditioned Q-less Cholesky QR} \label{sec:precond_cholqr}

The appeal of CholQR lies in its simplicity and parallelism.
\Cref{alg:cholqr} states the Q-less variant used throughout this paper. The Gram matrix can be assembled blockwise so that the tall intermediate $\mtx M[\mtx I]\in\R^{m\times n}$ is never formed. See \cref{sec:impl} for implementation details.

The fundamental limitation of CholQR is its numerical stability.
Forming the Gram matrix $\mtx{G}=\mtx{M}^{\trans}\mtx{M}$ squares the condition number of the input, $\cond(\mtx{G}) = \cond(\mtx{M})^2$, so the orthogonality loss of the implicitly defined $\mtx{Q} = \mtx{M}\mtx{R}^{-1}$ may be as large as $\mathcal{O}( u \,\cond(\mtx{M})^2)$~\cite{YamamotoNakatsukasa2015}.
When $\cond(\mtx{M}) > u^{-1/2}$ (roughly $10^8$ in double precision), $\mtx{G}$ may lose numerical positive definiteness entirely, causing the Cholesky factorization to~fail.

\begin{algorithm}[htb]
\small \setstretch{1.2}
\caption{Ordinary Q-less CholQR}
\label{alg:cholqr}
\begin{algorithmic}[1]
\Statex \textbf{Input:} linear operator $\mtx{M}$.
\Statex \textbf{Output:} upper-triangular $\mtx{R}$ such that $\mtx{M}\mtx{R}^{-1}$ has orthonormal columns.
\setstretch{1.25}
\Function{CholQR}{$\mtx{M}$}
\vspace{-4pt}
    \State Compute Gram matrix $\mtx{G} = \mtx{M}^{\trans}[\mtx{M}[\mtx I]]$
    \State Obtain Cholesky factorization $\mtx{G} = \mtx{R}^{\trans}\mtx{R}$
    \State \textbf{return} $\mtx{R}$
\EndFunction
\end{algorithmic}
\end{algorithm}

A natural remedy for the high rounding errors of CholQR is to apply it twice~\cite{FukayaEtAl2014}.
After the first pass produces $\mtx{R}_1$ with $\mtx{M}^{\trans}\mtx{M} = \mtx{R}_1^{\trans}\mtx{R}_1$, the preconditioned matrix $\mtx{Q}_1 = \mtx{M}\mtx{R}_1^{-1}$ satisfies $\cond(\mtx{Q}_1) \le 1 + \mathcal{O}(u \,\cond(\mtx{M})^2)$, which is $\mathcal O(1)$ whenever $\cond(\mtx M)\lesssim u^{-1/2}$.
A second CholQR pass on $\mtx{Q}_1$ then produces $\mtx{R}_2$ such that $\mtx{Q}_2 = \mtx{Q}_1\mtx{R}_2^{-1}$ has orthonormal columns to working precision.
This approach was formalized in~\cite{FukayaEtAl2014} as CholQR2 and proved to achieve $\mathcal{O}(u)$ orthogonality loss whenever $\cond(\mtx{M}) \lesssim u^{-1/2}$.
However, CholQR2 from~\cite{FukayaEtAl2014} does not admit an efficient Q-less implementation, because the second pass requires forming the explicit factor $\mtx{Q}_1$.
A Q-less variant follows from the more general preconditioning principle reresented by~\cref{alg:pcholqr}.

\begin{algorithm}[htb]
\small \setstretch{1.2}
\caption{Preconditioned Q-less CholQR}
\label{alg:pcholqr}
\begin{algorithmic}[1]
\Statex \textbf{Input:} linear operator $\mtx{M}$, nonsingular upper-triangular $\mtx{P} \in \R^{n \times n}$.
\Statex \textbf{Output:} upper-triangular $\mtx{R}$ such that $\mtx{M}\mtx{R}^{-1}$ has orthonormal columns.
\setstretch{1.25}
\Function{PCholQR}{$\mtx{M}, \mtx{P}$}
\vspace{-4pt}
    \State Define the inverse-preconditioner operator
    $\mtx{P}^{\mathrm{inv}}[\vct x] \leftarrow \mtx P^{-1}\vct x$
    and
    $\mtx{P}^{\mathrm{inv}\trans}[\vct x] \leftarrow \mtx P^{-\trans}\vct x$
    \label{line:cqrrt-pre}
    \State Define the preconditioned operator
   \Statex \hfill
    $\mtx{M}^{\mathrm{pre}}[\vct x]
    \leftarrow \mtx{M}[\mtx{P}^{\mathrm{inv}}[\vct x]]
    \quad\text{and}\quad
    \mtx{M}^{\mathrm{pre}\trans}[\vct x]
    \leftarrow
    \mtx{P}^{\mathrm{inv}\trans}[\mtx{M}^{\trans}[\vct x]]$
    \hfill\strut
    \State Perform Q-less CholQR on $\mtx{M}^\mathrm{pre}$:
    $\mtx R^\mathrm{chol} \leftarrow \mathtt{CholQR}(\mtx{M}^\mathrm{pre})$
    \label{line:cqrrt-gram}
    \State \textbf{return} $\mtx{R} = \mtx{R}^{\mathrm{chol}} \mtx{P}$
\EndFunction
\end{algorithmic}
\end{algorithm}

In other words, preconditioned CholQR is ordinary CholQR applied to the
redefined operator
$\mtx M^{\mathrm{pre}}=\mtx M\mtx P^{-1}$.
In practice, $\mtx{P}^{\mathrm{inv}}[\cdot]$ and
$\mtx{P}^{\mathrm{invT}}[\cdot]$ can be realized by solving triangular
systems with $\mtx P$ and $\mtx P^\trans$, respectively, at each
application.

\begin{remark}
As an alternative implementation, one may explicitly form
$\mtx{P}^{\mathrm{inv}}\approx \mtx P^{-1}$ and subsequently apply the
stored inverse. In this scenario,  $\mtx{P}^{\mathrm{inv}}$ should be obtained by solving
\[
\mtx P\mtx{P}^{\mathrm{inv}}=\mtx I
\]
columnwise using back substitution. The ordering of this solve is important. A columnwise solve of
$\mtx P\mtx{P}^{\mathrm{inv}}=\mtx I$ controls the residual
$\mtx P\mtx{P}^{\mathrm{inv}}-\mtx I.$
Its effect on the preconditioned operator can be written as
$
\mtx M\mtx P^{-1}-\mtx M\mtx{P}^{\mathrm{inv}}
=
-\mtx M\mtx P^{-1}
\bigl(\mtx P\mtx{P}^{\mathrm{inv}}-\mtx I\bigr).$
Thus, the inversion error enters as a perturbation of the already
preconditioned operator $\mtx M\mtx P^{-1}$, leading to stability
guarantees similar to those for implicit $\mtx{P}^{\mathrm{inv}}$.
\end{remark}

The additional cost of~\cref{alg:pcholqr} relative to~\cref{alg:cholqr} is essentially the cost of computing $\mtx P$. One natural choice for $\mtx P$ is the triangular factor $\mtx R_1$ obtained from a first CholQR pass, as in CholQR2. This yields a Q-less counterpart of CholQR2. More generally, the stability of the method is governed by the quality of $\mtx P$ as a preconditioner for $\mtx M$. We show in \cref{sec:theory} that using $\mtx P$ such that $\cond(\mtx M\mtx P^{-1})=\mathcal O(1)$ reduces the loss of orthogonality to $\mathcal O (u\,\cond(\mtx M))$.

CholQR2 still fails when $\cond(\mtx{M}) \gtrsim u^{-1/2}$, since the first Cholesky factorization may not succeed.
The method proposed in~\cite{FukayaEtAl2020} overcomes this barrier by introducing a diagonal shift: the first iteration factorizes $\mtx{G} + s\mtx{I}$ instead of $\mtx{G}$, where $s = c\,u\,n\,\|\mtx{M}\|_\mathrm{F}^2$ with $c$ chosen large enough to guarantee that the shifted Gram matrix is numerically positive definite.
The shift introduces a controlled perturbation: after the first, shifted pass, the condition number of $\mtx{Q}_1 = \mtx{M}\mtx{R}_1^{-1}$ is reduced to at most $\mathcal{O}(u^{-1/2})$, which is within the safe range for CholQR2. 
\Cref{alg:scholqr3} states the adaptation of the method to the Q-less setting. 

\begin{algorithm}[htb]
\small \setstretch{1.2}
\caption{Q-less sCholQR3}
\label{alg:scholqr3}
\begin{algorithmic}[1]
\Statex \textbf{Input:} linear operator $\mtx{M}$; shift parameter $s \ge 0$.
\Statex \textbf{Output:} upper-triangular $\mtx{R} \in \R^{n \times n}$ such that $\mtx{M}\mtx{R}^{-1}$ has orthonormal columns.
\setstretch{1.25}
\Function{sCholQR3}{$\mtx{M}, s$}
\vspace{-4pt}
    \State Compute Gram matrix $\mtx{G} = \mtx{M}^{\trans}[\mtx{M}[\mtx I]]$
    \State Obtain Cholesky factorization $\mtx{G} + s\mtx{I} = \mtx{R}_1^{\trans}\mtx{R}_1$.
    \For{$i = 2, 3$}
        \State $\mtx R_i \leftarrow \mathtt{PCholQR}{(\mtx{M}, \mtx{R}_{i-1})}$
    \EndFor
    \State \textbf{return} $\mtx{R} = \mtx{R}_3$
\EndFunction
\end{algorithmic}
\end{algorithm}

The resulting Q-less sCholQR3 method, however, essentially triples the cost of a single CholQR pass and does not fully resolve the numerical stability limitations of CholQR. In the Q-less setting, its admissible conditioning range extends only to $\cond(\mtx M)\lesssim u^{-3/4}$.\footnote{In fact, it falls short of the $\cond(\mtx M)\lesssim u^{-1}$ range attained by the sCholQR3 counterpart that explicitly forms the orthogonal factor.}
{This motivates the search for a more effective preconditioning strategy based on RandNLA, one capable of stabilizing Q-less CholQR throughout the entire numerically meaningful conditioning~range.}

\FloatBarrier

\section{Preconditioning via sketching}
\label{sec:precond_via_sk}

This section shows how RandNLA-based random sketching can overcome the fundamental
conditioning barrier of CholQR. Our CQRRTO algorithm, presented in
\cref{alg:cqrrt}, forms a random sketch
\[
\mtx M^{\rm sk}
=
\mtx S\mtx M
\in\mathbb R^{d\times n}, \text{with $d \ll m$,}
\]
It then uses the triangular factor of $\mtx M^{\rm sk}$ as the
preconditioner $\mtx P$ in \cref{alg:pcholqr}.

The effectiveness of random sketching can be quantified through the
classical framework of subspace
embeddings~\cite{Mah-mat-rev_BOOK,DM16_CACM,woodruff2014sketching}.
They were introduced in~\cite{DMM06}, first used in data-oblivious
form in~\cite{Sarlos06,DMMS07_FastL2_NM10}, and subsequently
popularized as a conceptual framework in~\cite{woodruff2014sketching}. 
Specifically, we seek $\mtx S$ that
approximately preserves the $\ell_2$-norms on the column space of $\mtx M$,
in the sense~that
\begin{equation}
\label{eq:embedding}
(1-\varepsilon_1)\|\mtx M\vct x\|_2^2
\le
\|\mtx S\mtx M\vct x\|_2^2
\le
(1+\varepsilon_2)\|\mtx M\vct x\|_2^2,
\qquad
\text{for all }\vct x\in\mathbb R^n,
\end{equation}
for some $\varepsilon_1\in[0,1)$ and $\varepsilon_2\ge0$.

A broad range of data-oblivious randomized constructions can be used for $\mtx S$,
including dense Gaussian or subgaussian matrices, randomized orthogonal
transforms, and sparse embeddings. For example, for any fixed
$\mtx M$ and $\varepsilon,\delta\in(0,1)$, dense subgaussian embeddings
satisfy~\cref{eq:embedding} with
$\varepsilon_1=\varepsilon_2=\varepsilon$ with probability at least
$1-\delta$, provided
\small 
\[
d
=
\mathcal O\!\left(
\varepsilon^{-2}
\bigl(n+\log(1/\delta)\bigr)
\right).
\]
\normalsize 
Sparse embeddings provide analogous guarantees while substantially
reducing the cost of applying $\mtx S$ \cite{MM13_STOC}. For fixed embedding accuracy and
failure probability, \cref{eq:embedding} can be satisfied by sparse
embeddings with sketching dimension $d=\mathcal O(n)$ and
$\mathtt{nnz}(\mtx S)=\mathcal O(m)$, up to logarithmic
factors~\cite{cohen2016nearly,chenakkod2024optimal}.
In practice, very sparse constructions can often be effective.
For example, recent experiments with SparseStack embeddings suggest that
a constant sparsity as small as $4$ nonzeros per row can provide reliable
sketches when $d=2n$~\cite{camano2025faster}. We likewise use sparse constructions in our numerical experiments and the RandLAPACK implementation, specifically short-axis-sparse sketching operators (SASOs)~\cite[\S 2.4.1]{MurrayEtAl2023}, in which every column of $\mtx S$ holds a fixed number $\zeta$ of nonzeros, placed at uniformly random rows with random signs and equal magnitudes.

Let $\mtx{R}^{\mathrm{sk}}$ be the triangular factor of $\mtx{S}\mtx{M}$, and let $\mtx{M}^{\mathrm{pre}} = \mtx{M}(\mtx{R}^{\mathrm{sk}})^{-1}$
be the corresponding preconditioned operator.
Writing $\vct y = \mtx{R}^{\mathrm{sk}}\vct x$ and using $\|\mtx{S}\mtx{M}\vct x\|_2 = \|\vct y\|_2$, the embedding property~\cref{eq:embedding} reads
\small 
\[
\frac{\|\mtx{M}^{\mathrm{pre}}\vct y\|_2^2}{\|\vct y\|_2^2}
=
\frac{\|\mtx{M}\vct x\|_2^2}{\|\mtx{S}\mtx{M}\vct x\|_2^2}
\in
\left[(1+\varepsilon_2)^{-1},\, (1-\varepsilon_1)^{-1}\right],
\]
\normalsize 
so the singular values of $\mtx{M}^{\mathrm{pre}}$ lie in $[(1+\varepsilon_2)^{-1/2},\,(1-\varepsilon_1)^{-1/2}]$ and
\small 
\begin{equation}\label{eq:precond-cond}
\cond(\mtx{M}^{\mathrm{pre}}) \leq \sqrt{\frac{1+\varepsilon_2}{1-\varepsilon_1}}.
\end{equation}
\normalsize
For practical embedding qualities (e.g., $\varepsilon_1 = 0.5$, $\varepsilon_2 = 1$), this gives $\cond(\mtx{M}^{\mathrm{pre}}) \leq 2$, well within the safe range for CholQR, and this holds regardless of $\cond(\mtx{M})$.

\begin{algorithm}[htb]
\small \setstretch{1.2}
\caption{Randomized preconditioned Q-less CholQR (CQRRTO)}
\label{alg:cqrrt}
\begin{algorithmic}[1]
\Statex \textbf{Input:} linear operator $\mtx{M}$, sketching dimension $d \geq n$, sketch density $\zeta$.
\Statex \textbf{Output:} upper-triangular $\mtx{R} \in \R^{n \times n}$ such that $\mtx{M}\mtx{R}^{-1}$ has orthonormal columns.
\setstretch{1.25}
\Function{CQRRTO}{$\mtx{M}, d, \zeta$}
\vspace{-4pt}
    \State Generate a random sparse embedding $\mtx{S} \in \R^{d \times m}$ with $\zeta$ nonzeros per column.
    \State Compute $\mtx{M}^{\mathrm{sk}} = \mtx{S}\mtx{M}[\mtx{I}]$ \label{line:cqrrt-sketch}
    \State Compute $[\sim, \mtx{R}^{\mathrm{sk}}] = \code{qr}(\mtx{M}^\mathrm{sk})$
    \State $\mtx R \leftarrow \mathtt{PCholQR}{(\mtx{M}, \mtx{R}^{\mathrm{sk}})}$
    \State \textbf{return} $\mtx{R}$
\EndFunction
\end{algorithmic}
\end{algorithm}

\begin{remark}[Sketched re-orthogonalization]
We also propose CQRRTO2 (\cref{alg:cqrrt2}), which adds a sketched
re-orthogonalization step to improve robustness for ill-conditioned
operators even more. It recomputes $\mtx R^{\rm sk}$ from a fresh application
of the preconditioned operator, either reusing the embedding or
resampling it independently. This additional pass increases the
operator application cost from $3nT_{\rm mv}$ to $4nT_{\rm mv}$.

\end{remark}

\begin{algorithm}[htb]
\small \setstretch{1.2}
\caption{CQRRTO with sketched re-orthogonalization (CQRRTO2)}
\label{alg:cqrrt2}
\begin{algorithmic}[1]
\Statex \textbf{Input:} linear operator $\mtx M$, sketching dimension $d\geq n$, sketch density $\zeta$.
\Statex \textbf{Output:} upper-triangular factor
$\mtx R\in\R^{n\times n}$.
\Function{CQRRTO2}{$\mtx M,d,\zeta$}
    \State Generate a random sparse embedding
        $\mtx S\in\R^{d\times m}$ with $\zeta$ nonzeros per column.
    \State $\mtx{M}^\mathrm{sk} \leftarrow \mtx S\mtx M[\mtx I]$, $[\sim,\mtx R^{\rm sk}]
        \leftarrow\code{qr}(\mtx{M}^\mathrm{sk})$.
    \State $\mtx P\leftarrow\mtx R^{\rm sk}$. Define $\mtx M^{\rm pre}[\mtx x]
        =\mtx M[\mtx P^{-1}\mtx x]$
        using triangular solves.
    \State (Optional) Resample $\mtx S$.
    \State $\mtx{M}^\mathrm{sk}           \leftarrow \mtx S\mtx M^{\rm pre}[\mtx I]$, $[\sim,\mtx R^{\rm sk}]
        \leftarrow\code{qr}(\mtx{M}^\mathrm{sk})$, $\mtx R^{\rm sk} \leftarrow \mtx R^{\rm sk} \mtx P$.
    \State $\mtx R\leftarrow
        \mathtt{PCholQR}(\mtx M, \mtx R^{\rm sk})$.
    \State \textbf{return} $\mtx R$.
\EndFunction
\end{algorithmic}
\end{algorithm}

\FloatBarrier

\section{Numerical stability analysis}
\label{sec:theory}

This section analyzes the loss of orthogonality in each Q-less QR
algorithm under finite-precision arithmetic. The analysis is built around
a simple observation: preconditioned Q-less CholQR in
\cref{alg:pcholqr} can be viewed as ordinary Q-less CholQR in
\cref{alg:cholqr} applied to the redefined operator
$\mtx M^{\rm pre}[\cdot]
:=
\mtx M\bigl[\mtx P^{\rm inv}[\cdot]\bigr].$
It therefore suffices to establish a single stability result for
CholQR under a sufficiently general model for finite-precision
operator applications.

To this end, \cref{subsec:operator_model} models rounding errors in the
forward and adjoint applications of an operator $\mtx M$ through input
and output error-shaping matrices, and shows how these shaping
matrices propagate under operator composition. This framework
preserves the structure and directional dependence of rounding errors,
which can lead to substantially sharper bounds than scalar normwise
majorants. 
\Cref{thm:cholqr-stability_first_order,thm:cholqr-stability} then establish the stability of
CholQR directly in terms of these error-shaping operators.
The subsequent result for preconditioned CholQR then follows immediately by applying these theorems to an
appropriate redefined operator.

For clarity, we here present the stability results with suppressed higher-order dependence on $u$. Complete versions of the key results, including all higher-order terms and tighter constants, together with their proofs, are given in~\cref{app:proofs}.

We use the standard componentwise floating-point arithmetic model~\cite{Higham2002}.
With $u$ the unit roundoff, we write
\small 
\[
\gamma_k:=\frac{ku}{1-ku},
\qquad
ku<1,
\]
\normalsize 
for the worst-case accumulation factor of an inner product of length $k$.
The analysis applies verbatim under any arithmetic model in which $\gamma_k$
is a valid error growth parameter. Under probabilistic rounding models,
for instance, $\gamma_k$ may be replaced by
$c\sqrt{k}\,u/(1-c\sqrt{k}\,u)$ for a suitable constant
$c>0$~\cite{HighamMary2019}.

\subsection{Rounding errors in operator applications}
\label{subsec:operator_model}

We characterize rounding errors in applications of $\mtx M$ using input and
output error-shaping operators, as defined below.

\begin{definition}
\label{def:operator-model}
A linear operator $\mtx M\in\R^{m\times n}$ satisfies the
\emph{operator rounding model} with error-shaping operators
$\mtx \Delta_{\rm out}^{M}$, $\mtx \Delta_{\rm in}^{M}$,
$\mtx \Delta_{\rm out}^{M^*}$, and $\mtx \Delta_{\rm in}^{M^*}$
if its forward and adjoint routines obey
\begin{subequations}
\begin{align}
&\mtx M[\mtx X]
=
\mtx M\mtx X
+
\mtx \Delta_{\rm out}^{M}\mtx E,
~
\|\mtx E\|_F
\le
u\| \mtx \Delta_{\rm in}^{M}\mtx X\|_F, \text{ and }
\label{eq:comp-matvec-M}
\\
&\mtx M^{\trans}[\mtx X]
=
\mtx M^{\trans}\mtx X
+
\mtx \Delta_{\rm out}^{M^*}\mtx E^*,
~
\|\mtx E^*\|_F
\le
u \|\mtx \Delta_{\rm in}^{M^*}\mtx X\|_F.
\label{eq:comp-matvec-MT}
\end{align}
\end{subequations}
\end{definition}
Here, $\mtx \Delta_{\rm in}^{M}$ and
$\mtx \Delta_{\rm in}^{M^*}$ characterize the dependence of the
rounding-error magnitude on the input, whereas
$\mtx \Delta_{\rm out}^{M}$ and
$\mtx \Delta_{\rm out}^{M^*}$ characterize how these errors are
propagated in the corresponding output spaces.

The examples below supply tools for determining error-shaping operators for  specific scenarios where the operators are black-box, explicit matrices, composite and/or involve solutions of linear systems.

\begin{example}[Black-box operator]
\label{ex:blackbox-operator}
For interpretability, we specialize some of our stability results to a
generic normwise black-box model. Specifically, suppose that the forward
and adjoint applications of $\mtx M$ satisfy, for all
$\vct x\in\R^n$ and $\vct y\in\R^m$,
\small 
\begin{equation}
\label{eq:blackbox-operator}
\|\mtx M[\vct x]-\mtx M\vct x\|_2
\le
Cu\,\|\mtx M\|_F\|\vct x\|_2,
\|\mtx M^\trans[\vct y]-\mtx M^\trans\vct y\|_2
\le
Cu\,\|\mtx M\|_F\|\vct y\|_2,
\end{equation}
\normalsize 
for some $C\ge0$.
This model satisfies \cref{def:operator-model}, for instance, with
the error-shaping operators $\mtx\Delta_{\rm out}^{M}=\mtx I_m,
\mtx\Delta_{\rm in}^{M}
=
C\|\mtx M\|_F\mtx I_n$, 
$\mtx\Delta_{\rm out}^{M^*}=\mtx I_n,
\mtx\Delta_{\rm in}^{M^*}
=
C\|\mtx M\|_F\mtx I_m.$
The black-box model is useful for exposing the overall dependence of
the stability bounds on the accuracy of the operator applications.
However, unlike the general error-shaping model, it does not capture
the interaction between the structure of the rounding errors and the
geometry of $\mtx M$. Consequently, the resulting dependence on $C$
can be pessimistic.
\end{example}
\begin{example}[Explicit matrix]
\label{ex:explicit-operator}
An explicit matrix $\mtx M\in\R^{m\times n}$ satisfies
\cref{def:operator-model}, for instance, with the normwise diagonal
error-shaping operators
\small 
\[
u\mtx\Delta_{\rm out}^{M}
=
\gamma_n \mtx I_m,
\qquad
u\mtx\Delta_{\rm in}^{M}
= 
u\sqrt n\,
\diag\!\bigl(
\|\vct m_1\|_2,\ldots,\|\vct m_n\|_2
\bigr),
\]
\normalsize 
and
\small 
\[
u\mtx\Delta_{\rm out}^{M^*}
=
u\diag\!\bigl(
\|\vct m_1\|_2,\ldots,\|\vct m_n\|_2
\bigr),
\qquad
u\mtx\Delta_{\rm in}^{M^*}
=
\sqrt n\,\gamma_m \mtx I_m,
\]
\normalsize 
where $\vct m_j$ denotes the $j$th column of $\mtx M$.
These choices follow directly from the standard componentwise
matrix-product error bounds.
\end{example}

\begin{example}[Inverse operator]
\label{lem:inverse-majorant}
Let $\mtx N\in\mathbb R^{n\times n}$ be nonsingular, and consider an
implementation of the inverse operator $\mtx N^{\rm inv}[\cdot]$ by a
backward-stable linear solver. Suppose that, for every right-hand side,
\[
\mtx N^{\rm inv}[\vct x]
=
(\mtx N+\delta\mtx N)^{-1}\vct x,
\qquad
|\delta\mtx N|
\le
u\delta_N|\mtx N|.
\]
Let $\mtx D_{\rm out}$ and $\mtx D_{\rm in}$ be arbitrary positive
diagonal matrices and suppose that
\[
u\delta_N
\left\|
\mtx D_{\rm out}^{-1}
|\mtx N|
\mtx D_{\rm in}^{-1}
\right\|_2
\left\|
\mtx D_{\rm in}\mtx N^{-1}\mtx D_{\rm out}
\right\|_2
<1.
\]
Then $\mtx N^{\rm inv}$ satisfies \cref{def:operator-model}. In
particular, valid error-shaping operators for the forward action are
\small 
\begin{equation}
\label{eq:inverse-majorant}
\mtx\Delta_{\rm out}^{N^{\rm inv}}
=
\mtx N^{-1}\mtx D_{\rm out},
\qquad
\mtx\Delta_{\rm in}^{N^{\rm inv}}
=
\frac{
\delta_N
\left\|
\mtx D_{\rm out}^{-1}
|\mtx N|
\mtx D_{\rm in}^{-1}
\right\|_2
}{
1-
u\delta_N
\left\|
\mtx D_{\rm out}^{-1}
|\mtx N|
\mtx D_{\rm in}^{-1}
\right\|_2
\left\|
\mtx D_{\rm in}\mtx N^{-1}\mtx D_{\rm out}
\right\|_2
}
\mtx D_{\rm in}\mtx N^{-1}.
\end{equation}
\normalsize
The diagonal matrices $\mtx D_{\rm out}$ and $\mtx D_{\rm in}$ may be
chosen as $\mtx I$.  
Alternatively, to sharpen the resulting error bounds, for example, they may be
selected to approximately equilibrate $|\mtx N|$. A similar
consideration applies to the adjoint inverse.
\end{example}

\begin{example}[Composition]
\label{lem:composition}
Let
\[
\mtx M[\cdot]
=
\mtx M_1[\mtx M_2[\cdot]],
\]
where each $\mtx M_i$ satisfies \cref{def:operator-model}.
Then $\mtx M$ also satisfies \cref{def:operator-model}. In particular,
for arbitrary scalars $\alpha_1,\alpha_2>0$, valid output/input
error-shaping operators for the forward action are
\small 
\begin{equation}
\label{eq:composite-forward-out}
{
\mtx\Delta_{\rm out}^{M}
=
\begin{bmatrix}
\alpha_1\mtx\Delta_{\rm out}^{M_1}
&
\alpha_2\mtx M_1\mtx\Delta_{\rm out}^{M_2}
\end{bmatrix},
\quad 
\mtx\Delta_{\rm in}^{M}
=
\begin{bmatrix}
\alpha_1^{-1}\mtx\Delta_{\rm in}^{M_1}\mtx M_2
\\[1mm]
\alpha_2^{-1}\mtx\Delta_{\rm in}^{M_2}
\end{bmatrix}
+
\mathcal O(u).
}
\end{equation}
\normalsize

Similarly, for arbitrary scalars $\alpha_1^*,\alpha_2^*>0$, valid
output/input error-shaping operators are
\small 
\begin{equation}
\label{eq:composite-adjoint-out}
{
\mtx\Delta_{\rm out}^{M^*}
=
\begin{bmatrix}
\alpha_1^*\mtx M_2^\trans
\mtx\Delta_{\rm out}^{M_1^*}
&
\alpha_2^*\mtx\Delta_{\rm out}^{M_2^*}
\end{bmatrix}, \quad \mtx\Delta_{\rm in}^{M^*}
=
\begin{bmatrix}
{\alpha_1^*}^{-1}\mtx\Delta_{\rm in}^{M_1^*}
\\[1mm]
{\alpha_2^*}^{-1}
\mtx\Delta_{\rm in}^{M_2^*}\mtx M_1^\trans
\end{bmatrix}
+
\mathcal{O}(u).
}
\end{equation}
\normalsize 
The free scaling parameters may be chosen to balance the input and
output error-shaping blocks.
\end{example}
The results from~\cref{ex:blackbox-operator,ex:explicit-operator,lem:inverse-majorant,lem:composition}  are particularly useful for tracking the propagation of local rounding errors through the preconditioned operator. They can also be combined or applied recursively to analyze more complex operator implementations, such as those considered in our numerical experiments.

\subsection{The preconditioned operator}
\label{subsec:preconditioned-operator}

Consider the composite preconditioned operator
$\mtx M^{\rm pre}=\mtx M\mtx P^{-1}$ from \cref{alg:pcholqr},
implemented as
\[
\mtx M^{\rm pre}[\cdot]
=
\mtx M[\mtx P^{\rm inv}[\cdot]],
\qquad
(\mtx M^{\rm pre})^\trans[\cdot]
=
\mtx P^{\rm invT}[\mtx M^\trans[\cdot]].
\]
In the default implementation of \cref{sec:precond_cholqr}, the action
of $\mtx P^{-1}$ is computed by solving a triangular system with
$\mtx P$ at each application. We model this computation by the
following assumption, attained with standard triangular solvers~\cite{Higham2002}.

\begin{assumption}[Inverse preconditioner]
\label{as:inverse-preconditioner}
Let $\mtx P\in\R^{n\times n}$ be nonsingular and upper triangular.
For each application, the computed inverse action satisfies
\[
\mtx P^{\rm inv}[\vct x]
=
(\mtx P+\delta\mtx P)^{-1}\vct x,
\qquad
|\delta\mtx P|
\le
\gamma_n|\mtx P|,
\]
where $\delta\mtx P$ may vary between applications. The adjoint action
$\mtx P^{\rm invT}[\cdot]$ satisfies the analogous backward-stability condition.
\end{assumption}

\begin{proposition}[Forward error shaping for the preconditioned operator]
\label{prop:preconditioned-forward-shapers}
Suppose that $\mtx M$ satisfies \cref{def:operator-model} and that
\cref{as:inverse-preconditioner} holds. If
$\gamma_n\cond_F(\mtx P)\le \frac12,$
then, up to terms of order $u^2$, valid error-shaping operators for the
forward realization
\[
\mtx M^{\rm pre}[\mtx X]
:=
\mtx M\bigl[\mtx P^{\rm inv}[\mtx X]\bigr]
\]
are
\small
\begin{equation}
\label{eq:trsm-preconditioned-shapers}
{
\mtx\Delta_{\rm out}^{M^{\rm pre}}
=
\begin{bmatrix}
\mtx M^{\rm pre}
&
\mtx\Delta_{\rm out}^{M}
\end{bmatrix},
~
u\mtx\Delta_{\rm in}^{M^{\rm pre}}
=
\begin{bmatrix}
2\gamma_n
\min\left \{\cond_F(\mtx P), \||\mtx P|\,|\mtx P^{-1}|\|_2 \right \} \mtx I
\\[1mm]
u\mtx\Delta_{\rm in}^{M}\mtx P^{-1}
\end{bmatrix}
+
\mathcal O(u^2).
}
\end{equation}
\normalsize 
The two blocks in~\cref{eq:trsm-preconditioned-shapers} correspond,
respectively, to applying the inverse preconditioner and applying
$\mtx M$.  The adjoint error-shaping operators follow analogously.
\end{proposition}

\begin{proof}
See~\cref{app:preconditioned-forward-proof}
\end{proof}

\begin{remark}[Explicit inverse preconditioner]
\label{rem:explicit-vs-trsm}
As a secondary implementation, one may form
$\mtx P^{\rm inv}$ explicitly and subsequently apply the stored inverse.
We consider forming $\mtx P^{\rm inv}$ by solving
$\mtx P\mtx P^{\rm inv}=\mtx I$
columnwise using a backward-stable triangular solver. Thus, for each
column $\vct e_j$ of the identity matrix, there exists a matrix
$\delta\mtx P_j$ such~that
\[
(\mtx P+\delta\mtx P_j)
\mtx P^{\rm inv}\vct e_j
=
\vct e_j,
\qquad
|\delta\mtx P_j|
\le
\gamma_n|\mtx P|.
\]
If $
\gamma_n
\bigl\||\mtx P|\,|\mtx P^{-1}|\bigr\|_2
\le \frac12,$
then, up to terms of order $u^2$, valid error-shaping operators for the
forward realization using the explicitly-stored inverse are
\small 
\begin{equation}
\label{eq:preconditioned-forward-shapers}
{
\mtx\Delta_{\rm out}^{M^{\rm pre}}
=
\begin{bmatrix}
\mtx M^{\rm pre}
&
\mtx M^{\rm pre}
&
\mtx\Delta_{\rm out}^{M}
\end{bmatrix},
\qquad
u\mtx\Delta_{\rm in}^{M^{\rm pre}}
=
\begin{bmatrix}
2\gamma_n
\bigl\||\mtx P|\,|\mtx P^{-1}|\bigr\|_2\mtx I
\\[2mm]
2\gamma_n
\bigl\||\mtx P|\,|\mtx P^{-1}|\bigr\|_2\mtx I
\\[2mm]
u\mtx\Delta_{\rm in}^{M}
\mtx P^{-1}
\end{bmatrix}
+
\mathcal O(u^2).
}
\end{equation}
\normalsize 
The three blocks correspond, respectively, to forming the inverse
preconditioner, applying the explicitly-stored inverse, and applying
$\mtx M$.
\end{remark}

In both implementations, rounding errors associated with the inverse
preconditioner are propagated through the composite operator
$\mtx M^{\rm pre}=\mtx M\mtx P^{-1}$, leading to the same qualitative
stability behavior.

\subsection{Stability of Q-less Cholesky QR}
\label{subsec:main_thm}

\begin{assumption}[Roundoff regime]
\label{as:roundoff}
Throughout \cref{sec:theory}, we assume
\begin{equation}
\label{eq:standing-roundoff}
0\le \gamma_n\le 0.2\,\gamma_m\le 0.002 .
\end{equation}

We further assume that the Cholesky factorization applied to a matrix $\widehat{\mtx G}$ first symmetrizes it with no additional rounding error.
With a slight abuse of notation, we henceforth denote this symmetrized matrix again by
$\widehat{\mtx G}$. We assume that the Cholesky factorization of
$\widehat{\mtx G}$ succeeds and satisfies
\begin{equation}
\label{eq:comp-chol-error}
(\widehat{\mtx R}^{\rm chol})^{\trans}
\widehat{\mtx R}^{\rm chol}
=
\widehat{\mtx G}+\mtx E^{\rm chol},
\qquad
|\mtx E^{\rm chol}|
\le
\gamma_{n+1}
|\widehat{\mtx R}^{\rm chol}|^{\trans}
|\widehat{\mtx R}^{\rm chol}|,
\end{equation}
whenever
\[
1.01n\,\cond(\widehat{\mtx G})
\le 1/\gamma_{n+1},
\]
which is a standard assumption; see, e.g., \cite{FukayaEtAl2020}.
\end{assumption}

The following theorem provides the basic stability result underlying
the analysis of all Q-less CholQR variants considered in this work.
The bound is expressed directly in terms of the input and output
error-shaping operators of $\mtx M$, thereby retaining the structure
of its particular implementation.  The complete result
including explicit higher-order terms in the operator-error contribution $u_M$, is given in~\cref{app:proofs_cholqr}.

\begin{theorem}[Stability of Q-less CholQR (first-order form)]
\label{thm:cholqr-stability_first_order}
Let $\mtx M\in\R^{m\times n}$ have full column rank and suppose that it
satisfies \cref{def:operator-model}. Let
\[
\mtx M=\mtx Q\mtx R
\]
denote its exact QR factorization. Define
\begin{equation}
\label{eq:cholqr-operator-error}
\begin{aligned}
u_M
:={}
u
\frac{\|\mtx Q^\trans\mtx\Delta_{\rm out}^{M}\|_2}
     {\|\mtx M\|_F}
\|\mtx\Delta_{\rm in}^{M}\|_F
+
u
\|\mtx R^{-\trans}\mtx\Delta_{\rm out}^{M^*}\|_2\|\mtx\Delta_{\rm in}^{M^*}\mtx Q\|_2
+ \mathcal{O}(u^2).
\end{aligned}
\end{equation}
Consider \cref{alg:cholqr} and set
\[
\widehat{\mtx Q}
:=
\mtx M\widehat{\mtx R}^{-1}.
\]
If, in addition to the roundoff regime described by
\cref{as:roundoff},
\begin{equation}
\label{eq:cholqr-smallness}
1.05\,u_M\cond_F(\mtx M)
+
1.5\,\gamma_{n+1}\cond_F(\mtx M)^2
\le0.2, \quad 
1.6\,n
\gamma_{n+1}\cond_F(\mtx M)^2
\le1,
\end{equation}
then the Cholesky factorization succeeds, $\widehat{\mtx R}$ is nonsingular and
\begin{equation}
\label{eq:cholqr-Q-bound}
\boxed{
\left\|
\widehat{\mtx Q}^{\trans}\widehat{\mtx Q}-\mtx I
\right\|_F
\le
1.25\,u_M\cond_F(\mtx M)
+
1.75\,\gamma_{n+1}\cond_F(\mtx M)^2 .
}
\end{equation}
\end{theorem}

\begin{proof}
See \cref{app:proofs_cholqr}.
\end{proof}

The following corollary specializes \cref{thm:cholqr-stability_first_order} to
generic black-box implementations of $\mtx M$. 

\begin{corollary}[Stability of Q-less CholQR for black-box operators]
\label{cor:cholqr-blackbox}
Let $\mtx M\in\R^{m\times n}$ have full column rank and suppose that perturbations in 
its forward  and adjoint applications satisfy the black-box operator assumptions from~\cref{ex:blackbox-operator}.
Then the coefficient $u_M$ in \cref{thm:cholqr-stability_first_order,thm:cholqr-stability} satisfies
\[
u_M
\le
1.2\,C u
\left(
\sqrt{n}+\cond_F(\mtx M)
\right).
\]
Consequently, under the hypotheses of \cref{thm:cholqr-stability_first_order},
\[
\left\|
\widehat{\mtx Q}^{\trans}\widehat{\mtx Q}-\mtx I
\right\|_F
\le
1.5\,C u
\left(
\sqrt{n}+\cond_F(\mtx M)
\right)
\cond_F(\mtx M)
+
1.75\,\gamma_{n+1}\cond_F(\mtx M)^2 .
\]
\end{corollary} 

\begin{proof}
Substituting bounds~\cref{eq:blackbox-operator} into
\cref{eq:cholqr-operator-error-full} in \cref{app:proofs_cholqr} gives
$u_M=Cu(\sqrt n+\cond_F(\mtx M))
+C^2u^2\sqrt n\,\cond_F(\mtx M)$.
Under \cref{eq:cholqr-smallness-full}, $Cu\sqrt n\le0.2$, so the
claimed coefficient bound follows, and \cref{thm:cholqr-stability}
yields the orthogonality estimate.
\end{proof}

\Cref{thm:cholqr-stability_first_order,cor:cholqr-blackbox} provide a loss scaling $\|\widehat{\mtx Q}^{\trans}\widehat{\mtx Q}-\mtx I\|_F
=
\mathcal O\!\left(Cu\,\cond(\mtx M)^2\right).$
The quadratic dependence on $\cond_F(\mtx M)$ represents a fundamental
stability limitation of CholQR. 

We note that for structured operator implementations the dependence of the loss on
$C$ can be substantially milder than~\Cref{cor:cholqr-blackbox} would suggest. The error-shaping model captures this
through the structured quantities appearing in $u_M$. In particular,~\Cref{cor:cholqr-blackbox} uses the conservative unstructured bound for the dominant~term
\[
u\|\mtx R^{-\trans}\mtx\Delta_{\rm out}^{M^*}\|_2
\|\mtx\Delta_{\rm in}^{M^*}\mtx Q\|_2 \leq u\|\mtx R^{-\trans}\|_2
\|\mtx\Delta_{\rm out}^{M^*}\|_2
\|\mtx\Delta_{\rm in}^{M^*}\|_2
\leq
Cu\,\cond_F(\mtx M). 
\]
For favorable operator structures, this term can remain 
$\mathcal O(Cu)$ and provide the overall~bound 
\[
\|\widehat{\mtx Q}^{\trans}\widehat{\mtx Q}-\mtx I\|_F = \mathcal O\!\left(Cu\,\cond_F(\mtx M)\right)
+
\mathcal O\!\left(u\,\cond_F(\mtx M)^2\right).
\]
However,  even in this favorable regime  the fundamental $u^{-1/2}$ conditioning barrier of an
unpreconditioned CholQR pass is still retained.

\subsection{Stability of preconditioned Q-less Cholesky QR}
\label{subsec:pcholqr-stability}

The stability of preconditioned Q-less CholQR follows directly from
\cref{thm:cholqr-stability_first_order,thm:cholqr-stability} by viewing \cref{alg:pcholqr} as ordinary
CholQR applied to the composite operator
$\mtx M^{\rm pre}=\mtx M\mtx P^{-1}$.
Under \cref{as:inverse-preconditioner}, the rounding errors
introduced by forming and applying the inverse preconditioner are
captured by the error-shaping operators derived in
\cref{subsec:preconditioned-operator}. The complete result
including the higher-order terms, is given in~\cref{thm:pcholqr-stability}.

\begin{theorem}[Stability of preconditioned Q-less CholQR (first-order form)]
\label{thm:pcholqr-stability_first_order}
Let $\mtx M\in\R^{m\times n}$ have full column rank and suppose that it
satisfies \cref{def:operator-model}. Let $\mtx M=\mtx Q\mtx R$
be the exact QR factorization of $\mtx M$. Let
$\mtx P\in\R^{n\times n}$ be a nonsingular upper-triangular
preconditioner whose inverse actions satisfy
\cref{as:inverse-preconditioner}, and suppose that $\gamma_n\cond_F(\mtx P)\le\frac12.$
Define
\[
\mtx M^{\rm pre}:=\mtx M\mtx P^{-1}.
\]
Consider \cref{alg:pcholqr} and let
\[
\widehat{\mtx Q}
:=
\mtx M\widehat{\mtx R}^{-1},
\qquad
\widehat{\mtx R}
=
\widehat{\mtx R}^{\rm chol}\mtx P,
\]
where the matrix products are interpreted in exact arithmetic. Define
\small 
\begin{equation}
\label{eq:pcholqr-upre}
\begin{aligned}
u_{\rm pre}
&:={}
2.5\,\gamma_n\cond_F(\mtx P)
\left(
\sqrt n+\cond_2(\mtx M^{\rm pre})
\right)
+
1.25\,u
\|\mtx Q^\trans\mtx\Delta_{\rm out}^{M}\|_2
\|\mtx\Delta_{\rm in}^{M}\mtx R^{-1}\|_2\\
&+
1.25\,u
\|\mtx R^{-\trans}\mtx\Delta_{\rm out}^{M^*}\|_2
\|\mtx\Delta_{\rm in}^{M^*}\mtx Q\|_2
+
\mathcal O(u^2).
\end{aligned}
\end{equation}
\normalsize 
If, in addition to the roundoff regime of 
\small \cref{as:roundoff},
\begin{equation}
\label{eq:pcholqr-smallness}
1.05\,u_{\rm pre}\cond_F(\mtx M^{\rm pre})
+
1.5\,\gamma_{n+1}\cond_F(\mtx M^{\rm pre})^2
\le 0.2,
~~
1.6\,n\gamma_{n+1}
\cond_F(\mtx M^{\rm pre})^2
\le 1,
\end{equation}
\normalsize 
then the Cholesky factorization succeeds,
$\widehat{\mtx R}^{\rm chol}$ is nonsingular, and
\begin{equation}
\label{eq:pcholqr-Q-bound}
\boxed{
\left\|
\widehat{\mtx Q}^{\trans}\widehat{\mtx Q}-\mtx I
\right\|_F
\le
1.25\,u_{\rm pre}\cond_F(\mtx M^{\rm pre})
+
1.75\,\gamma_{n+1}
\cond_F(\mtx M^{\rm pre})^2 .
}
\end{equation}
\end{theorem}

\begin{proof}
See \cref{app:proofs_pcholqr}.
\end{proof}

The following corollary specializes
\cref{thm:pcholqr-stability_first_order} to generic black-box
implementations of $\mtx M$.

\begin{corollary}[Stability of preconditioned Q-less CholQR for
black-box operators]
\label{cor:pcholqr-blackbox}
Let $\mtx M\in\R^{m\times n}$ have full column rank and suppose that
the perturbations in its forward and adjoint applications satisfy black-box operator assumptions from~\cref{ex:blackbox-operator}.
Under the assumptions and conditions of
\cref{thm:pcholqr-stability}, the coefficient $u_{\rm pre}$ satisfies
\[
u_{\rm pre}
\le
2.5\,\gamma_n\cond_F(\mtx P)
\left(
\sqrt n+\cond_2(\mtx M^{\rm pre})
\right)
+
2.81\,Cu\,\cond_F(\mtx M).
\]
Consequently,
\small 
\[
\boxed{
\begin{aligned}
\left\|
\widehat{\mtx Q}^{\trans}\widehat{\mtx Q}-\mtx I
\right\|_F
\le{}&
3.125\,\gamma_n\cond_F(\mtx P)
\left(
\sqrt n+\cond_2(\mtx M^{\rm pre})
\right)
\cond_F(\mtx M^{\rm pre})
\\
&+
3.513\,Cu\,
\cond_F(\mtx M)
\cond_F(\mtx M^{\rm pre})
\\
&+
1.75\,\gamma_{n+1}
\cond_F(\mtx M^{\rm pre})^2 .
\end{aligned}
}
\]
\normalsize 
\end{corollary}

\begin{proof}
See \cref{app:proofs_pcholqr}.
\end{proof}

The significance of
\cref{cor:pcholqr-blackbox,thm:pcholqr-stability_first_order}
is that the conditioning entering the CholQR stability bound is that of
the preconditioned operator $\mtx M^{\rm pre}$ rather than that of
$\mtx M$. The dependence on the preconditioner is represented solely through $\cond(\mtx M^{\rm pre})$ and $\cond(\mtx P)$. In particular, if $\mtx P$ is such
that
\[
\cond(\mtx M^{\rm pre})=\mathcal O(1),
\qquad
\cond(\mtx P)=\mathcal O(\cond(\mtx M)),
\]
then the quadratic dependence on $\cond(\mtx M)$ of ordinary CholQR
is reduced to a linear dependence. The dependence on
$\cond(\mtx P)$ reflects rounding errors incurred in applying the
inverse preconditioner.

For structured operator implementations, the dependence on the
black-box error coefficient $C$ can be substantially milder than
suggested by \cref{cor:pcholqr-blackbox}. This behavior is captured by
the error-shaping quantities retained in
\cref{thm:pcholqr-stability_first_order}.

\begin{remark}
\label{rem:accumulated-product}
\Cref{thm:pcholqr-stability} provides a bound for the exact product of triangular factors
$\mtx R = \widehat{\mtx R}^{\rm chol}\mtx P$. If the final factor is formed explicitly
in floating-point arithmetic, the rounding error of this additional
matrix product must be included separately. Its analysis is standard.
\end{remark}

\subsection{Stability of Q-less Cholesky QR2}
\label{subsec:stab_cholqr2}

We now specialize \cref{thm:pcholqr-stability} to Q-less CholQR2,
which performs two consecutive CholQR passes: the first produces the
triangular factor $\widehat{\mtx R}^{\rm chol}_1$, which is then used
as the preconditioner $\mtx P=\widehat{\mtx R}^{\rm chol}_1$ in the
second pass. For simplicity, we state the result for black-box operators.

\begin{corollary}[Stability of Q-less CholQR2 for black-box operators]
\label{cor:qless-cholqr2-blackbox}
Let $\mtx M\in\R^{m\times n}$ have full column rank and suppose that
its forward and adjoint applications satisfy the black-box operator
assumptions~\cref{eq:blackbox-operator}. Define
\[
u_{\rm qr2}
:=
2.77\,\gamma_n
\left(
\sqrt n+1.11
\right)
+
2.81\,Cu.
\]

Suppose that, in addition to the roundoff regime of
\cref{as:roundoff},
\begin{equation}
\label{eq:qr2-blackbox-first-smallness}
\left(
1.25\,u_{\rm qr2}
+
1.75\,\gamma_{n+1}
\right)
\cond_F(\mtx M)^2
\le0.1,
\qquad
1.6\,n\gamma_{n+1}\cond_F(\mtx M)^2
\le1.
\end{equation}
Suppose further that
\begin{equation}
\label{eq:qr2-blackbox-second-smallness}
1.17\,\sqrt n\,u_{\rm qr2}\cond_F(\mtx M)
+
1.84\,n\gamma_{n+1}
\le0.2,
\qquad
1.96\,n^2\gamma_{n+1}
\le1.
\end{equation}
Then both Cholesky factorizations succeed, the final triangular factor
$\widehat{\mtx R}$ is nonsingular, and
\begin{equation}
\label{eq:qr2-blackbox-final-bound}
\boxed{
\left\|
\widehat{\mtx Q}^{\trans}\widehat{\mtx Q}-\mtx I
\right\|_F
\le
1.39\,\sqrt n\,u_{\rm qr2}\cond_F(\mtx M)
+
2.14\,n\gamma_{n+1}.
}
\end{equation}
\end{corollary}

\begin{proof}
See \cref{app:proofs_cholqr2}.
\end{proof}

The second pass removes one power of the condition number from the loss
of orthogonality. In particular, for fixed dimensions and black-box
coefficient $C$, \cref{eq:qr2-blackbox-final-bound} yields the
leading-order behavior
$\|
\widehat{\mtx Q}^{\trans}\widehat{\mtx Q}-\mtx I
\|_F
=
\mathcal O\!\left(u\,\cond_F(\mtx M)\right).$
The admissible conditioning range, however, remains limited by the
first pass. Since $\widehat{\mtx R}^{\rm chol}_1$ is obtained by
ordinary CholQR, the initial Gram matrix must remain numerically
positive definite. Consequently,
\cref{eq:qr2-blackbox-first-smallness} retains the characteristic
conditioning restriction
\[
\cond_F(\mtx M)\lesssim u^{-1/2}.
\]

\subsection{Stability of Q-less shifted Cholesky QR3}
\label{subsec:stab_scholqr3}

To lift the conditioning restriction of CholQR2,
\cref{alg:scholqr3} replaces the first ordinary CholQR pass by a
shifted factorization
\begin{equation}
\label{eq:schol}
\widehat{\mtx G}+s\mtx I
=
\widehat{\mtx R}_1^\trans\widehat{\mtx R}_1,
\qquad
s=c\,u\,n\|\mtx M\|_F^2.
\end{equation}
A sufficiently large shift guarantees that the first Cholesky
factorization succeeds without requiring
$\cond(\mtx M)\lesssim u^{-1/2}$. The price is that the shifted factor
does not fully orthogonalize $\mtx M$, and the conditioning of the
intermediate operator
$\mtx M\widehat{\mtx R}_1^{-1}$ deteriorates proportionally to the
shift relative to $\sigma_{\min}(\mtx M)^2$.
The following lemma makes this precise for black-box operators.

\begin{lemma}[Conditioning after the shifted pass for black-box operators]
\label{lem:shifted-first-pass}
Let $\mtx M\in\R^{m\times n}$ have full column rank and suppose that
its forward and adjoint applications satisfy the black-box operator
assumptions~\cref{eq:blackbox-operator}. Consider
\cref{alg:scholqr3} with shift $s=c\,u\,n\|\mtx M\|_F^2,$ where $c$ is a user-defined constant. 
Assume, in addition to the roundoff regime described by
\cref{as:roundoff}, that
\begin{equation}
\label{eq:shifted-pass-smallness}
2.1\,Cu\sqrt n\le0.2,
\qquad
n\gamma_{n+1}\le\frac16,
\qquad
cun\ge
\max\left\{
4\left(2.1\,Cu\sqrt n+1.5\,\gamma_{n+1}\right),
\,2n\gamma_{n+1}
\right\}.
\end{equation}
Then the shifted Cholesky factorization succeeds and
$\widehat{\mtx R}_1$ is nonsingular. Moreover,
\begin{equation}
\label{eq:shift-R-bounds}
\sigma_{\min}(\widehat{\mtx R}_1)^2
\ge
0.55\,s,
\qquad
\|\widehat{\mtx R}_1\|_F^2
\le
\frac{
1.2+cun^2
}{
1-\gamma_{n+1}
}
\|\mtx M\|_F^2,
\end{equation}
and
\begin{equation}
\label{eq:shift-cond}
\sigma_{\max}(\mtx M\widehat{\mtx R}_1^{-1})
\le1,
\qquad
\sigma_{\min}(\mtx M\widehat{\mtx R}_1^{-1})^{-1}
\le
\sqrt{
1+1.45\,c\,u\,n\,\cond_F(\mtx M)^2
}.
\end{equation}
\end{lemma}

\begin{proof}
See \cref{app:proofs_cholqr3}.
\end{proof}

This yields the following stability result for Q-less sCholQR3.

\begin{corollary}[Stability of Q-less sCholQR3 for black-box operators]
\label{cor:scholqr3-stability}
Let $\mtx M\in\R^{m\times n}$ have full column rank, with $n\ge10$,
and suppose that its forward and adjoint applications satisfy the
black-box operator assumptions of \cref{ex:blackbox-operator}.
Consider \cref{alg:scholqr3} with shift $s=c\,u\,n\|\mtx M\|_F^2,$
and define coefficient 
\begin{equation}
\label{eq:scholqr3-beta}
\beta
:=
1+1.45\,c\,u\,n\,\cond_F(\mtx M)^2.
\end{equation}
Define further
\small 
\begin{equation}
\label{eq:scholqr3-phi}
u_{\rm qr3}
:=
\sqrt{
\frac{2.22}{c\,u}
+
1.85\,n^2
}
\left[
(0.51\sqrt n+2.8)\gamma_n
+
2.2(1+\gamma_n)Cu
\right].
\end{equation}
\normalsize 

Assume that the hypotheses of \cref{lem:shifted-first-pass} hold and,
in addition,
\begin{equation}
\label{eq:scholqr3-smallness}
n\beta
\left[
u_{\rm qr3}
+
1.5\,\gamma_{n+1}
\right]
\le0.2,
\qquad
1.44\,n
\left[
\sqrt{\beta}\,u_{\rm qr3}
+
1.5\,\gamma_{n+1}
\right]
\le0.2,
\end{equation}
and
\begin{equation}
\label{eq:scholqr3-chol-success}
1.9\,n^2\beta\,\gamma_{n+1}
\le1.
\end{equation}
Then all three Cholesky factorizations succeed, $\widehat{\mtx R}$ is nonsingular, and
\begin{equation}
\label{eq:scholqr3-final-bound}
\boxed{
\left\|
\widehat{\mtx Q}^{\trans}\widehat{\mtx Q}-\mtx I
\right\|_F
\le
1.5\,n\sqrt{\beta}\,u_{\rm qr3}
+
2.52\,n\,\gamma_{n+1},
}
\end{equation}
where $\widehat{\mtx Q}
:=
\mtx M\widehat{\mtx R}^{-1}$.
\end{corollary}

\begin{proof}
See \cref{app:proofs_cholqr3}.
\end{proof}

Once the shifted pass has succeeded, the two subsequent preconditioned
passes recover the stability of Q-less CholQR2, but over the larger
admissible range prescribed by \cref{eq:scholqr3-smallness}. In the
strongly ill-conditioned regime,
$\beta=\mathcal O(u\,\cond(\mtx M)^2)$ and
$u_{\rm qr3}=\mathcal O(u^{1/2})$.
Consequently, the dominant condition in \cref{eq:scholqr3-smallness,eq:scholqr3-final-bound} scale as
\[
\beta u_{\rm qr3}
=
\mathcal{O}\left(
u^{3/2}\cond(\mtx M)^2
\right),\quad 
\sqrt{\beta} u_{\rm qr3}
=
\mathcal{O}\left(
u\cond(\mtx M)
\right)
\]
and yield the admissible range
\begin{equation}
\label{eq:scholqr3-admissible-range}
\cond(\mtx M)
\lesssim
u^{-3/4}.
\end{equation}
Thus, Q-less sCholQR3 extends the
$u^{-1/2}$ admissible range of Q-less CholQR2 to
$u^{-3/4}$, but it still does not reach the full numerically meaningful regime $\cond(\mtx M)\lesssim u^{-1}$.

\begin{remark}[Why the shifted Q-less method stops at $u^{-3/4}$]
\label{rem:shift-vs-sketch}
The limitation comes from the conditioning of the operator produced by
the shifted pass. By \cref{eq:shift-cond},
\[
\sigma_{\min}(\mtx M\widehat{\mtx R}_1^{-1})
\ge
\beta^{-1/2},
\]
and, in the strongly ill-conditioned regime,
$\beta=\mathcal O(u \, \cond(\mtx M)^2)$. Thus, although the shift
guarantees the success of the first Cholesky factorization, it leaves
an intermediate operator whose conditioning can grow as
$\mathcal O(u^{1/2}\cond(\mtx M))$. 

In contrast, as will be shown below, an effective randomized preconditioner can make
$\mtx M\mtx P^{-1}$ well conditioned directly, without first passing
through such an ill-conditioned intermediate operator. This removes
the $\beta$ factor from the corresponding stability condition and extends the admissible range to
$\cond(\mtx M)\lesssim u^{-1}$.
\end{remark}

\subsection{Stability of randomized Q-less Cholesky QR}
\label{subsec:stab_CQRRTO}

CQRRTO uses the triangular factor
$\widehat{\mtx R}^{\rm sk}$ of a random sketch of $\mtx M$ as the
preconditioner. Its stability therefore relies on two properties of
the computed sketch: the exact sketch must satisfy the embedding property, and the errors incurred in
forming and factorizing the sketch must be sufficiently small. 

\begin{assumption}[Computed sketched preconditioner]
\label{as:sketch}
The computed sketch satisfies
\begin{equation}
\label{eq:sketch-formation-model}
\widehat{\mtx M}^{\rm sk}
=
\mtx S\mtx M+\mtx E^{\rm sk},
\qquad
\|\mtx E^{\rm sk}\|_2
\le
u_{\rm sk}\|\mtx M\|_2,
\end{equation}
and its computed triangular factor $\widehat{\mtx R}^{\rm sk}$
satisfies
\[
\mtx Q^{\rm sk}\widehat{\mtx R}^{\rm sk}
=
\widehat{\mtx M}^{\rm sk}+\mtx E^{\rm qr},
\]
for some $\mtx Q^{\rm sk}$ with orthonormal columns, where
\begin{equation}
\label{eq:computed-sketch-qr-backward}
\|\mtx E^{\rm qr}\|_2
\le
u_{\rm qr}\|\widehat{\mtx M}^{\rm sk}\|_2,
\qquad
u_{\rm qr}\le0.01.
\end{equation}
\end{assumption}

\begin{remark}[Verifying \cref{as:sketch}]
\label{rem:sketch-constants}
For a backward-stable QR factorization,
\cref{eq:computed-sketch-qr-backward} holds with the usual
backward-error constants. For example, Householder QR satisfies
\[
\|\mtx E^{\rm qr}\|_F
\le
\gamma_{6dn}\|\widehat{\mtx M}^{\rm sk}\|_F
\]
for $\widehat{\mtx M}^{\rm sk}\in\R^{d\times n}$
\cite[Ch.~19]{Higham2002}, and hence one may take
\[
u_{\rm qr}
=
\sqrt n\,\gamma_{6dn}.
\]
The sketch-formation coefficient $u_{\rm sk}$ can similarly be
bounded from the rounding-error model of the particular implementation
of $\mtx S\mtx M$. We retain $u_{\rm sk}$ explicitly, since this
bound depends on the sketching operator and on how the sketch is
formed.
\end{remark}

The following lemma shows that the embedding property is inherited by
the computed triangular factor.

\begin{lemma}[Conditioning of the computed sketched preconditioner]
\label{lem:computed-sketch-conditioning}
Let $\mtx M\in\R^{m\times n}$ have full column rank and suppose that
$\mtx S$ is a subspace embedding for $\range(\mtx M)$ satisfying
\begin{equation}
\label{eq:computed-sketch-embedding}
(1-\varepsilon_1)\|\mtx M\vct x\|_2^2
\le
\|\mtx S\mtx M\vct x\|_2^2
\le
(1+\varepsilon_2)\|\mtx M\vct x\|_2^2,
\qquad
\vct x\in\R^n,
\end{equation}
where $0\le\varepsilon_1<1$ and $\varepsilon_2\ge0$.
Suppose further that \cref{as:sketch} holds and assume that 
\begin{equation}
\label{eq:computed-sketch-rho-def}
\rho_{\rm sk}
:=
\left[
(1+u_{\rm qr})u_{\rm sk}
+
u_{\rm qr}\sqrt{1+\varepsilon_2}
\right]
\cond_2(\mtx M)<\sqrt{1-\varepsilon_1}.
\end{equation}
Then define
\small 
\begin{equation}
\label{eq:CQRRTO-tau}
\tau
:=
\frac{
\sqrt{1+\varepsilon_2}+\rho_{\rm sk}
}{
\sqrt{1-\varepsilon_1}-\rho_{\rm sk}
}.
\end{equation}
\normalsize 
Then, with $\mtx P:=\widehat{\mtx R}^{\rm sk}$,
\begin{equation}
\label{eq:CQRRTO-preconditioner-bounds}
\cond_2(\mtx M\mtx P^{-1})\le\tau,
\qquad
\cond_F(\mtx M\mtx P^{-1})\le\sqrt n\,\tau,
\qquad
\cond_F(\mtx P)\le\tau\cond_F(\mtx M).
\end{equation}
\end{lemma}

\begin{proof}
See \cref{app:proofs_CQRRTO}.
\end{proof}

We can now obtain the stability of CQRRTO directly from
\cref{cor:pcholqr-blackbox}. For clarity, we state the stability result only for operators satisfying the black-box model~\cref{eq:blackbox-operator}. The extension to the general error-shaping model of \cref{def:operator-model} follows directly from \cref{thm:pcholqr-stability_first_order}.

\begin{corollary}[Stability of CQRRTO for black-box operators]
\label{cor:computed-sketched-preconditioner-stability}
Let $\mtx M\in\R^{m\times n}$ have full column rank and suppose that
its forward and adjoint applications satisfy the black-box rounding-error
bounds~\cref{eq:blackbox-operator}. Consider \cref{alg:cqrrt} under
the sketch-formation and QR backward-error model of \cref{as:sketch},
and suppose that the sketch preserves norms on $\range(\mtx M)$
as specified in \cref{eq:computed-sketch-embedding}.
Let $\rho_{\rm sk}$ and $\tau$ be defined by
\cref{eq:computed-sketch-rho-def,eq:CQRRTO-tau}, with
$\rho_{\rm sk}<\sqrt{1-\varepsilon_1}$ so that the computed sketch
retains full column rank.
Suppose that the inverse actions of
$\mtx P=\widehat{\mtx R}^{\rm sk}$ satisfy the triangular-solve
backward-error bounds of \cref{as:inverse-preconditioner}, and that the product used to
accumulate the triangular factors is interpreted in exact arithmetic.
Define
\begin{equation}
\label{eq:CQRRTO-u}
u_{\rm CQRRTO}
:=
2.5\,\gamma_n\tau
\left(
\sqrt n+\tau
\right)
+
2.81\,Cu.
\end{equation}
Suppose that, in addition to the roundoff regime of
\cref{as:roundoff},
\begin{equation}
\label{eq:CQRRTO-smallness}
1.05\,\sqrt n\,\tau\,
u_{\rm CQRRTO}\cond_F(\mtx M)
+
1.5\,n\tau^2\gamma_{n+1}
\le0.2, \quad 
1.6\,n^2\tau^2\gamma_{n+1}
\le1.
\end{equation}
Then the Cholesky factorization succeeds, the returned
factor $\widehat{\mtx R}$ is nonsingular, and
\begin{equation}
\label{eq:computed-sketch-Q-bound}
\boxed{
\left\|
\widehat{\mtx Q}^{\trans}\widehat{\mtx Q}-\mtx I
\right\|_F
\le
1.25\,\sqrt n\,\tau\,
u_{\rm CQRRTO}\cond_F(\mtx M)
+
1.75\,n\tau^2\gamma_{n+1},
}
\end{equation}
where $\widehat{\mtx Q}
:=
\mtx M\widehat{\mtx R}^{-1}$.
\end{corollary}

\begin{proof}
See \cref{app:proofs_CQRRTO}.
\end{proof}

For typical sketching matrices, the embedding distortions
$\varepsilon_1$ and $\varepsilon_2$ are bounded by constants with high
probability. Provided that the errors in forming and factorizing the
sketch are sufficiently small so that $\rho_{\rm sk}$ remains bounded away
from $\sqrt{1-\varepsilon_1}$, this gives $
\tau=\mathcal O(1)$
independently of $\cond(\mtx M)$. For fixed dimensions and black-box
coefficient $C$, \cref{eq:CQRRTO-u} therefore gives $u_{\rm CQRRTO}=\mathcal O(u),$
and the dominant stability condition
\cref{eq:CQRRTO-smallness} scales as
$$u\,\cond_F(\mtx M)\lesssim1.$$
Thus, CQRRTO extends the $u^{-1/2}$ admissible range of Q-less CholQR
and CholQR2, and the $u^{-3/4}$ range of Q-less sCholQR3, to the full
numerically meaningful regime $\cond_F(\mtx M)\lesssim u^{-1}$.

Similar stability guarantees apply to CQRRTO2
(\cref{alg:cqrrt2}).
Furthermore, since the second sketch is formed from an already preconditioned
operator, it is expected to yield a smaller effective $\tau$,
improving the constants and extending the admissible conditioning
range within the same $\mathcal O(u^{-1})$ regime.
\section{Experiments}
\label{sec:experiments}

This section has two parts. \Cref{subsec:synthetic} characterizes the stability
and downstream accuracy of Q-less QR methods on small synthetic problems
in MATLAB.
\Cref{subsec:ls_benchmarks} evaluates triangular factors computed by
the RandLAPACK implementations from \cref{sec:impl} as right preconditioners
for iterative least-squares solvers on matrix-free finite-element
and Toeplitz benchmarks. 

\subsection{Synthetic numerical benchmarks}
\label{subsec:synthetic}

\subsubsection{Orthogonalization stability}
The analysis in \cref{sec:theory} predicts how the loss of
orthogonality depends on the conditioning of $\mtx M$ and the errors
in its application. We test these predictions using two families of
synthetic operators. 

\paragraph{Black-box operator.}
We form and store
\[
\mtx M = \mtx U\mtx D\mtx V^\trans,
\]
where $\mtx U\in\R^{m\times n}$ and $\mtx V\in\R^{n\times n}$
have Haar-distributed orthonormal columns, and $\mtx D$ is diagonal
with entries logarithmically spaced between $1$ and $1/\kappa$,
so that $\cond(\mtx M)=\kappa$. The algorithms access $\mtx M$
through the routines
\[
\mtx M[\mtx X]
=
\mtx M\mtx X
+
Cu\|\mtx M\|_F\,\mtx E_{\rm f}\mtx X,
\qquad
\mtx M^\trans[\mtx Y]
=
\mtx M^\trans\mtx Y
+
Cu\|\mtx M\|_F\,\mtx E_{\rm a}\mtx Y,
\]
where $\mtx E_{\rm f}$ and $\mtx E_{\rm a}$ are obtained from
independent standard Gaussian matrices by normalizing each to unit
spectral norm. They are drawn once and held fixed. These perturbations
follow the black-box model of \cref{ex:blackbox-operator}. We vary
$C$ over six orders of magnitude.

\paragraph{Shaped operator.}
We also consider the matrix-free composition
\[
\mtx M
=
\mtx T_1^{-1}\mtx A\mtx T_3^{-1}
\mtx W\mtx D\mtx T_2^{-1},
\]
where $\mtx T_1\in\R^{m\times m}$ and
$\mtx T_2\in\R^{n\times n}$ are sparse, well-conditioned unit
triangular matrices, $\mtx A\in\R^{m\times n}$ is sparse with
$\cond(\mtx A)\approx\kappa_A$, and $\mtx T_3$ is the triangular
factor in a thin QR factorization of $\mtx A$. The matrix $\mtx W$
is orthogonal, and $\mtx D$ is diagonal with all entries equal to $1$ except five interior entries equal to $1/\kappa$, so that $\cond(\mtx D)=\kappa$. $\mtx M$ is never formed explicitly and is applied
through matrix products and triangular solves.

Since $\mtx A\mtx T_3^{-1}$ has orthonormal columns, the conditioning
of $\mtx M$ is determined by $\mtx D$, up to the effect of the
well-conditioned factors $\mtx T_1$ and $\mtx T_2$. Thus,
$\cond(\mtx M)\approx\kappa$ even as $\kappa_A$ increases.
However, $\mtx A$ and $\mtx T_3^{-1}$ each have condition number
approximately $\kappa_A$, so applying them separately can amplify
rounding errors. Bounds based only on the overall error norm can
therefore become increasingly pessimistic as $\kappa_A$ grows.
The error-shaping bounds retain information about how these errors
propagate through the composition and can give sharper estimates.

\paragraph{Setup.}
All experiments use double precision with $m=500$ and $n=40$.
For each target value of $\kappa$, we generate $10$ independent
operators and report the maximum of each error measure over the draws.
We compare Q-less CholQR (\cref{alg:cholqr}), Q-less CholQR2
(\cref{alg:pcholqr}, using the triangular factor from the first
CholQR pass as the preconditioner), Q-less sCholQR3
(\cref{alg:scholqr3}), CQRRTO (\cref{alg:cqrrt}) and  CQRRTO2 (\cref{alg:cqrrt2}).
For sCholQR3, we use the shift constant $c=11$. Tests with $c=100$
and $c=1000$ gave comparable or worse accuracy. CQRRTO uses a sparse
sign embedding with sketch dimension $d=2n$ and $\zeta=8$ nonzeros
per column. If a Cholesky factorization fails, we return the identity
matrix as the triangular factor.

We measure the loss of orthogonality
\[
\|\mtx Q^\trans\mtx Q-\mtx I\|_2,
\qquad
\mtx Q=\mtx M\widehat{\mtx R}^{-1},
\]
where $\mtx Q$ is computed using triangular solves.

\paragraph{Loss of orthogonality: black-box operator.}
\Cref{fig:synth-blackbox-orth} shows the loss of orthogonality as
$\cond(\mtx M)$ increases, for perturbation levels
$C=0$, $10^2$, $10^4$, and $10^6$.

\begin{figure}[htb]
    \centering
    \includegraphics[width=\textwidth]{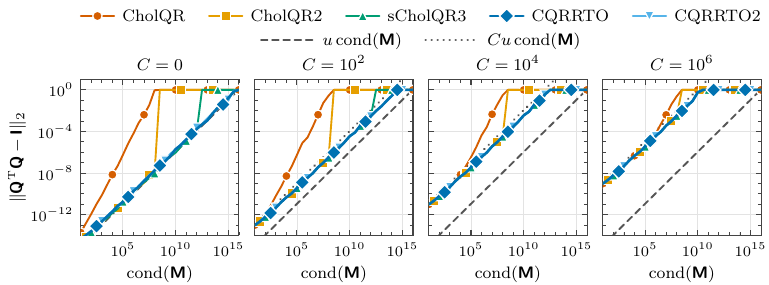}
    \caption{Loss of orthogonality $\|\mtx Q^\trans\mtx Q-\mtx I\|_2$
    against $\cond(\mtx M)$ for the black-box operator, with
    $C=0$, $10^2$, $10^4$, and $10^6$ from left to right.
    For $C>0$, the dotted
    line shows the reference scaling $Cu\,\cond(\mtx M)$.
    }
    \label{fig:synth-blackbox-orth}
\end{figure}

For $C=0$, the results are consistent with
\cref{tab:stability-summary}. CholQR exhibits quadratic error growth
and reaches order-one loss of orthogonality near
$\cond(\mtx M)\approx3\times10^8$, consistent with the
$u^{-1/2}$ stability scale. In contrast, CholQR2, sCholQR3, and
CQRRTO show approximately linear growth, following
$u\,\cond(\mtx M)$ over their respective stability ranges.
The main difference is the conditioning range they can handle:
CholQR2 is limited to about $10^9$ by its unpreconditioned first
pass, sCholQR3 extends the range to about $10^{13}$, and CQRRTO
reaches about $3\times10^{15}$, approaching the $u^{-1}$ scale.

For $C>0$, the CQRRTO error follows
$Cu\,\cond(\mtx M)$ to within a small factor. Thus, increasing the operator error
raises the observed error level proportionally, consistent with
\cref{cor:computed-sketched-preconditioner-stability}.

\paragraph{Loss of orthogonality: shaped operator.}
\Cref{fig:synth-shaped-orth} repeats the experiment for the matrix-free
family with $\kappa_A=1,10^2,10^4,10^6$. 
\begin{figure}[htb]
    \centering
    \includegraphics[width=\textwidth]{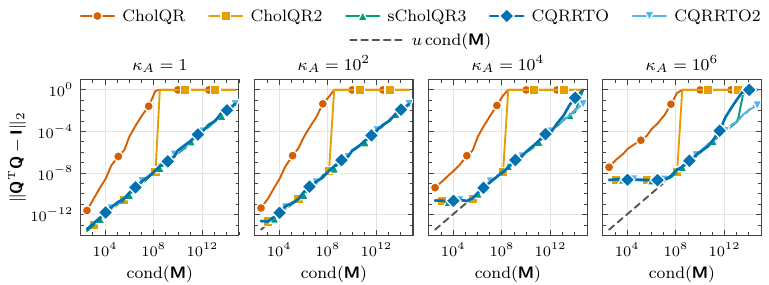}
    \caption{Loss of orthogonality $\|\mtx Q^\trans\mtx Q-\mtx I\|_2$
    against $\cond(\mtx M)$ for the shaped operator
    $\mtx T_1^{-1}\mtx A\mtx T_3^{-1}\mtx W\mtx D\mtx T_2^{-1}$,
    with $\kappa_A=1$, $10^2$, $10^4$, and $10^6$ from left to right.}
    \label{fig:synth-shaped-orth}
\end{figure}

The stabilized methods again show approximately linear error growth
over their respective stability ranges. CholQR2 remains limited by
its first pass, while sCholQR3 and CQRRTO remain accurate over a much
wider conditioning range. Increasing $\kappa_A$ raises the error
plateau at small $\cond(\mtx M)$, reflecting the larger errors in
the operator applications. Above this plateau, the curves remain
close over much of the conditioning range, although additional
deterioration appears at the largest condition numbers, particularly
for $\kappa_A=10^6$.

Unlike the black-box perturbations, these application errors inherit
structure from the factors defining $\mtx M$. The error-shaping bounds account
for this interaction, explaining why larger application errors need
not lead to a proportional increase in the loss of orthogonality.
We compare the two types of bounds next.

\paragraph{Validation of the bounds.}
\Cref{fig:synth-bounds} compares the measured loss of orthogonality
for CholQR and CQRRTO with the first-order estimates from
\cref{thm:cholqr-stability_first_order,thm:pcholqr-stability_first_order}.
We evaluate each estimate in two forms: a shaped version, which
retains the error-shaping operators, and a black-box version, which
replaces them by scalar normwise bounds.

We retain the numerical constants from \cref{sec:theory}, except
that we replace $\gamma_k$ by $\sqrt{k}u/(1-\sqrt{k}u)$, following the well-known square-root rule of thumb for rounding-error accumulation~\cite{Higham2002}.
We choose the free scaling parameters in the error-shaping model
to minimize each estimate.

\begin{figure}[htb]
    \centering
    \includegraphics[width=\textwidth]{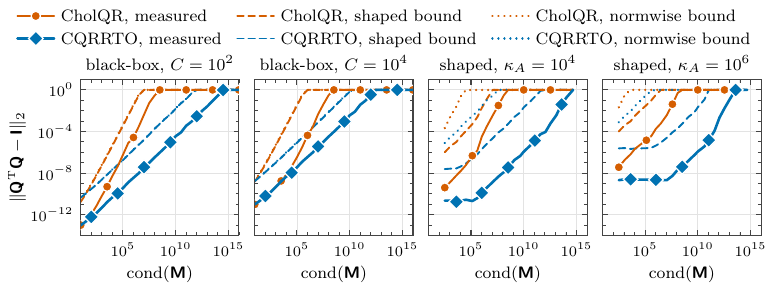}
    \caption{Measured loss of orthogonality for CholQR and CQRRTO
    compared with the first-order estimates from \cref{sec:theory}.
    From left to right: the black-box operator with $C=10^2$ and
    $C=10^4$, and the shaped operator with $\kappa_A=10^4$ and
    $\kappa_A=10^6$. Solid curves show measured errors, dashed curves
    show bounds with error-shaping model, and dotted curves show bounds with
    normwise model.}
    \label{fig:synth-bounds}
\end{figure}

For the black-box operator, the two versions nearly coincide at
both perturbation levels. The error model contains no additional
directional information, so retaining the error-shaping operators
offers little improvement over
\cref{cor:cholqr-blackbox,cor:pcholqr-blackbox}. 

For the shaped operator, the error-shaping estimates are smaller
than the black-box estimates by several orders of magnitude over
much of the conditioning range. They remain conservative but
follow the measured behavior closely.

\paragraph{Realization of the inverse preconditioner.}
In \cref{alg:cqrrt}, we apply the inverse of the triangular
preconditioner $\mtx P$ by back substitution, without forming
$\mtx P^{-1}$. We compare this implementation with three ways
of forming and storing an approximate inverse $\mtx P^{\rm inv}$:
columnwise back substitution in
$\mtx P\mtx P^{\rm inv}=\mtx I$, rowwise substitution in
$\mtx P^{\rm inv}\mtx P=\mtx I$, and blockwise inversion with
LAPACK's \code{TRTRI}. These constructions are equivalent in exact
arithmetic, but their finite-precision behavior differs substantially.

\begin{figure}[htb]
    \centering
    \includegraphics[width=0.6\textwidth]{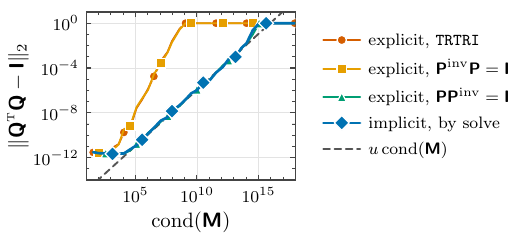}
    \caption{Loss of orthogonality for CQRRTO on the shaped operator
    with different realizations of the inverse preconditioner.}
    \label{fig:synth-precond-inverse}
\end{figure}

In \cref{fig:synth-precond-inverse}, implicit application and
columnwise inverse construction give nearly identical results,
with approximately linear error growth above the initial plateau.
The other two constructions lose accuracy much faster, reaching
order-one loss of orthogonality near $\cond(\mtx M)\approx10^9$,
compared with $10^{15}$--$10^{16}$ for the first two methods.

The distinction is which inverse residual is controlled.
Columnwise back substitution controls
$\mtx P\mtx P^{\rm inv}-\mtx I$, as described in
\cref{rem:explicit-vs-trsm}. This residual determines the
perturbation introduced by forming the inverse:
$\mtx M\mtx P^{\rm inv}
-
\mtx M\mtx P^{-1}
=
\mtx M\mtx P^{-1}
\bigl(\mtx P\mtx P^{\rm inv}-\mtx I\bigr).$
The rowwise and blockwise constructions instead control the
oppositely ordered residual
$\mtx P^{\rm inv}\mtx P-\mtx I$, which need not provide
comparable control when $\mtx P$ is ill-conditioned.
Thus, this demonstration confirms that the stability of CQRRTO depends not only on the choice of
preconditioner, but also on how its inverse action is realized.

\subsubsection{Downstream accuracy and robustness}
We next examine whether improved orthogonalization
stability translates into more accurate and robust downstream
computations from~\cref{subsec:apps}.
The first three tests use the black-box operator
with $C=10^2$. The fourth uses a separate operator construction.

\emph{Least squares.}
We use $\widehat{\mtx R}$ as a right preconditioner for
\cref{eq:lsq} and perform two Richardson refinement steps on the
preconditioned normal equations. We report forward error
$\frac{\|\vct x-\vct x_\star\|_2}
{\|\vct x_\star\|_2+\|\vct x\|_2},$
where $\vct x_\star$ is a solution computed by
Householder QR.

\emph{Orthogonal projection.}
We apply the projector onto $\range(\mtx M)$ using the Q-less
representation in \cref{subsec:apps} and report
$\max_{\vct v}
\frac{\|\mtx\Pi[\vct v]-\mtx\Pi_\star\vct v\|_2}
{\|\vct v\|_2},$
where the maximum is taken over random test vectors and
$\mtx\Pi_\star$ is formed from an explicit orthonormal basis.

\emph{Implicit normal update.}
We test the Woodbury-based computation in \cref{subsec:apps}
with $\mtx V$ the black-box synthetic operator and
$\mtx B =s^2\mtx I$, where $s=\sigma_{\min}(\mtx V)$.
Each Q-less QR method is applied to the augmented operator
$\mtx M=[\mtx V^\trans/s\,\mtx I]^\trans$.
Here, ordinary Q-less CholQR corresponds to the standard evaluation of the Woodbury formula: it evaluates
the capacitance matrix
$\mtx I+\mtx V^\trans\mtx V/s^2$ and computes the Cholesky factorization to applying its inverse.
We report
$\max_{\vct v}
\frac{\|\mtx\Sigma[\vct v]-\mtx\Sigma_\star\vct v\|_2}
{\|\vct v\|_2/s^2+\|\mtx\Sigma[\vct v]\|_2},$
where the maximum is taken over five independent standard
Gaussian test vectors per trial and ten trials, and
$\mtx\Sigma_\star$ is a reference for
$(s^2\mtx I+\mtx V\mtx V^\trans)^{-1}$
computed from the SVD of $\mtx V$.

\emph{Reduced spectral approximation.}
To isolate the effect of basis conditioning on Rayleigh--Ritz
extraction, we construct $\mtx M=\mtx U\mtx D\mtx V^\trans$,
where the columns of $\mtx U$ are the eigenvectors associated
with the $n$ smallest eigenvalues of the one-dimensional
Dirichlet Laplacian $\mtx C$. The remaining factors are chosen
as before, so that $\cond(\mtx M)=\kappa$.
Thus, $\range(\mtx M)$ is an invariant subspace of $\mtx C$,
and the exact Ritz values are $\lambda_1,\ldots,\lambda_n$.
We assemble the reduced matrix
$\mtx H
=
\widehat{\mtx R}^{-\trans}
\mtx M^\trans [\mtx C\,\mtx M[\widehat{\mtx R}^{-1}]]$
using operator applications and triangular solves, without forming
$\mtx Q$. We report the normalized eigenvalue error
$\frac{\max_i|\theta_i-\lambda_i|}
{\max_i|\lambda_i|+\max_i|\theta_i|},$
where $\theta_i$ are the eigenvalues of the computed reduced matrix,
with both $\theta_i$ and $\lambda_i$ ordered increasingly.

\begin{figure}[htb]
    \centering
    \includegraphics[width=\textwidth]{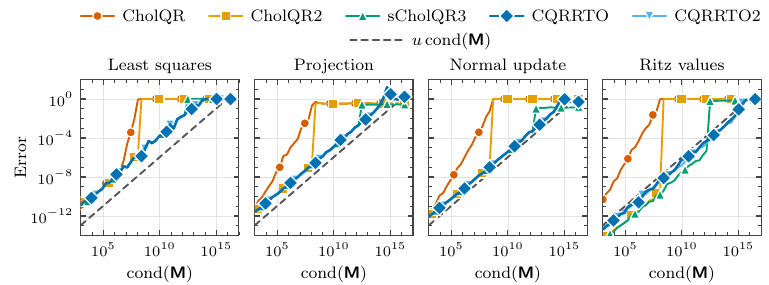}
    \caption{Downstream accuracy against $\cond(\mtx M)$.
    From left to right: least-squares error, orthogonal-projection
    error, implicit normal-update error, and Ritz-value error.
    The first three tests use the black-box operator with $C=10^2$.
    The fourth uses an invariant subspace of the Dirichlet
    Laplacian, with its corresponding eigenvalues as the
    reference Ritz values.}
    \label{fig:synth-downstream}
\end{figure}

The downstream errors follow the main trends observed in the
orthogonality experiments. CholQR loses accuracy first, while
the stabilized methods generally exhibit approximately linear
error growth over their respective stability ranges. CQRRTO
retains useful accuracy over the widest conditioning range. 
These results
show that the improved stability of the orthogonalization
translates into more accurate/robust downstream computations.

\subsection{Least-squares benchmarks}
\label{subsec:ls_benchmarks}

We next compare the Q-less methods with Blendenpik~\cite{Blendenpik2010}
and an unpreconditioned least-squares solver, on realistic operators. CQRRTO2 is not included
because CQRRTO already produces an almost optimal preconditioner. All
sketching-based methods use the same embedding distribution and sketch
dimensions. The configuration of every method is given in
\cref{subsec:compared_algs}.

The benchmarks solve $\min_{\vct x}\|\mtx M\vct x-\vct b\|_2$,
with $\vct b=\mtx M\vct x_{\rm true}+\vct e$.
The entries of $\vct x_{\rm true}$ are independent and uniform
in $[-1,1]$, and $\vct e$ is Gaussian noise scaled to satisfy
$\|\vct e\|_2=10^{-11}\|\mtx M\vct x_{\rm true}\|_2$.

\paragraph{Shift retries.}
In our RandLAPACK implementation, we enable a shift-and-retry
strategy similar to \cite[Algorithm~2]{BindhakPellingSaak2026}:
if a Cholesky factorization fails, a diagonal shift is added
or increased, and the factorization attempt is repeated until success (see \cref{subsec:blocked_gram}).
This improves the robustness of CholQR and CholQR2 on ill-conditioned
problems. We leave its theoretical analysis to future work.

\paragraph{Downstream solver.}
Both benchmarks use conjugate gradients on the preconditioned normal
equations within an iterative-refinement loop, following the approach
of \cite{EpperlyMeierNakatsukasa2024}. Each outer cycle recomputes
$\vct r=\vct b-\mtx M\vct x$ in working precision and restarts
conjugate gradients to compute a correction.

\paragraph{Stopping criteria.}
In both benchmarks, each outer cycle solves the preconditioned normal
equations for the correction, with right-hand side
$\mtx R^{-\trans}\mtx M^\trans\vct r$, by CG started from zero.
This inner solve stops when its CG residual falls to $10^{-4}$ of its
starting value, after $20$ consecutive iterations that reduce it by less
than $0.1\%$ (the best iterate is then kept), or after $500$ iterations.
For the FEM benchmark, the inner solve also stops once its CG residual
falls below $10^{-16}\|\mtx R^{-\trans}\mtx M^\trans\vct b\|_2$, so that cycles
at the roundoff level are not iterated. The outer loop stops
when the relative residual error $\|\vct b - \mtx M\vct x\|_2/\|\vct b\|_2$ reaches $10^{-12}$, fails to improve
over two consecutive cycles, or after $50$ cycles.
For the Toeplitz benchmark, every five inner iterations and after each outer
cycle, we estimate the Karlson--Wald\'en least-squares backward error using
\cite[eq.~(4.2)]{EpperlyMeierNakatsukasa2024} with an independent
sparse embedding of $2n$ rows and $8$ nonzeros per column.
The solve stops as soon as this estimate, normalized by
$\|\mtx M\|_F$, falls below $\sqrt n\,u$; the cost of the checks is included in the reported times.
As safeguards, we also apply the FEM outer stopping criteria.

All benchmark runs use the same
64-core Intel Xeon node; the hardware, the timing protocol and the configuration
of every method are given in \cref{sec:ls_set,subsec:compared_algs}.

\subsubsection{FEM least-squares benchmark}
\label{subsec:results_fem2}
We consider discretized Petrov--Galerkin operators
$\mtx M=\mtx L^{-1}\mtx K\mtx U$ 
for a 2D advection-diffusion-reaction PDE with highly
heterogeneous coefficients. Here, $\mtx K$ is the stiffness matrix,
$\mtx L$ is the Cholesky factor of the corresponding mass matrix,
and $\mtx U$ is a sparse coarse-grid basis. Forward and adjoint
applications are matrix-free, using sparse products and triangular
solves.  

\Cref{fig:benchmarks} (left half) and \cref{tab:fem-observations}
report results for a $75824\times8304$ operator with
$\cond(\mtx M)\approx5\times10^{14}$.
\begin{figure}[htb]
    \centering
    \includegraphics[width=\textwidth]{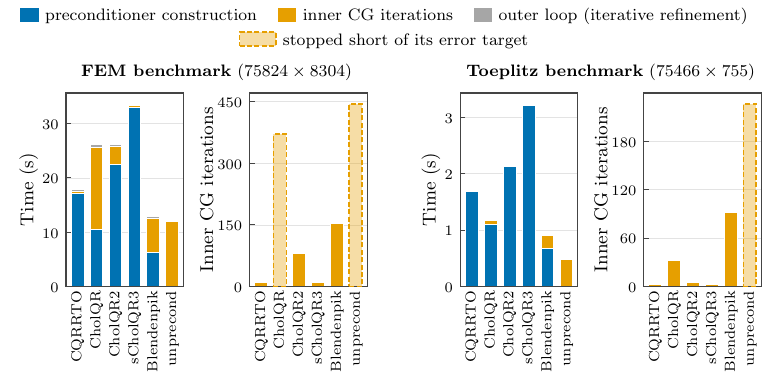}
    \caption{Wall-clock time and inner CG iterations for the FEM benchmark (left half), and for
    the Toeplitz benchmark (right half ). Iteration counts are totaled over all refinement cycles. A lighter
    bar marks a method that stopped without reaching the error target.}
    \label{fig:benchmarks}
\end{figure}
\begin{table}[htb]
\caption{FEM benchmark results. Factorization time is relative to
CQRRTO's. 
``Floor'' indicates a residual error within 5\% of the optimal value. Orthogonality loss is $\|{\widehat{\mtx Q}}^\trans\widehat{\mtx Q}-\mtx I\|_F/\sqrt n$,
where $\widehat{\mtx Q}=\mtx M[\mtx I]\mtx R^{-1}$ is formed only for this diagnostic.}
\label{tab:fem-observations}
\centering
\footnotesize
\setlength{\tabcolsep}{4pt}
\begin{adjustbox}{max width=\textwidth}
\begin{tabular}{@{}lrrrrr@{}}
\toprule
Method & \shortstack[r]{Shift\\retries} & \shortstack[r]{Orthogonality\\loss} & \shortstack[r]{Inner CG\\iterations} & \shortstack[r]{Factorization\\time} & \shortstack[r]{Residual \\error} \\
\midrule
CQRRTO           & $0$   & $2.6\times10^{-4}$ & $9$   & $1.00$ & floor \\
sCholQR3         & $0$   & $9.2\times10^{-5}$ & $9$   & $1.92$ & floor \\
CholQR2          & $2$   & $0.11$             & $82$  & $1.31$ & floor \\
CholQR & $1$   & $0.91$             & $371$ & $0.61$ & $1.3\times10^{-11}$ \\
Blendenpik       & n/a   & $0.71$             & $153$ & $0.36$ & floor \\
unpreconditioned & n/a & n/a           & $445$ & n/a    & $3.1\times10^{-3}$ \\
\bottomrule
\end{tabular}
\end{adjustbox}
\end{table}
CQRRTO requires no shift and attains an orthogonality loss consistent
with the $u\cond(\mtx M)$ scaling in \cref{tab:stability-summary}.
It reaches the error floor in $9$ inner CG iterations.
sCholQR3 reaches the same accuracy in the same number of iterations
but takes $1.92$ times as long to construct the preconditioner.
CholQR2 requires two shift retries and $82$ iterations.
CholQR has a cheaper factorization but requires $371$ iterations
and stagnates above the residual error floor, making its total time
more than $40\%$ higher than CQRRTO's.
Blendenpik also reaches the floor, but requires $153$ iterations. The results suggest that CQRRTO offers a total runtime advantage
over Blendenpik when the preconditioner is reused for as few as
two or three right-hand sides.
The unpreconditioned solver stagnates at a relative residual error
of $3.1\times10^{-3}$.

Tests on two additional operators with different sizes and condition
numbers are reported in \cref{subsec:extra_cases} and show similar trends.

\subsubsection{Prolate Toeplitz least-squares benchmark}
\label{subsec:results_toeplitz}
We consider the Toeplitz least-squares problem 
with the regularized operator $\mtx M=[\mtx T^\trans \;\lambda\mtx I]^\trans$,
where $\mtx T$ is a prolate Toeplitz operator with bandwidth
parameter $\omega=0.14$ and $\lambda=10^{-10}$.
Forward and adjoint applications of $\mtx T$ are matrix-free,
using FFTs of a circulant embedding.

We consider right-hand
sides of the form $\vct b = [\vct t^\trans \; \vct 0]^\trans$. \Cref{fig:benchmarks} (right half) and
\cref{tab:toeplitz-observations} report results for a
$75466\times755$ Toeplitz operator.

\begin{table}[htb]
\caption{Toeplitz benchmark results.
Factorization time is relative to
CQRRTO's. 
Orthogonality loss is $\|{\widehat{\mtx Q}}^\trans \widehat{\mtx Q}-\mtx I\|_F/\sqrt n$,
where $\widehat{\mtx Q}=\mtx M[\mtx I]\mtx R^{-1}$ is formed only for this diagnostic.
Backward error is estimated from the sketch at termination and normalized
by $\|\mtx M\|_F$.
``Floor'' indicates an estimated backward error on the order of $10^{-15}$.}
\label{tab:toeplitz-observations}
\centering
\footnotesize
\setlength{\tabcolsep}{4pt}
\begin{adjustbox}{max width=\textwidth}
\begin{tabular}{@{}lrrrrr@{}}
\toprule
Method & \shortstack[r]{Shift\\retries} & \shortstack[r]{Orthogonality\\loss} & \shortstack[r]{Inner CG\\iterations} & \shortstack[r]{Factorization\\time} & \shortstack[r]{Backward\\error} \\
\midrule
CQRRTO           & $0$ & $2.6\times10^{-6}$ & $3$   & $1.00$ & floor \\
sCholQR3         & $0$ & $1.4\times10^{-6}$ & $3$   & $1.90$ & floor \\
CholQR2          & $1$ & $9.8\times10^{-5}$ & $5$   & $1.26$ & floor \\
CholQR           & $1$ & $0.84$             & $32$  & $0.65$ & floor \\
Blendenpik       & n/a & $0.72$             & $92$  & $0.40$ & floor \\
unpreconditioned & n/a & n/a                & $226$ & n/a    & $1.2\times10^{-12}$ \\
\bottomrule
\end{tabular}
\end{adjustbox}
\end{table}

All preconditioned methods reach the backward-error target, differing mainly in cost. Only CholQR
and CholQR2 require shift retries. Factorization dominates
the total time of the Q-less methods. CQRRTO's factorization
takes $0.53$ times as long as sCholQR3's and $0.79$ times as
long as CholQR2's.
CholQR and Blendenpik have cheaper factorizations but require
$32$ and $92$ iterations,
respectively. The unpreconditioned solver has the shortest
total time but stagnates before reaching the target.
Tests on three additional operator sizes in \cref{subsec:extra_cases}
show similar trends.

\section{Concluding remarks}\label{sec:conclusion}
We developed several RandNLA-based stabilized Q-less QR methods for tall matrix-free
operators. Our main algorithm, CQRRTO, uses a randomized sketch
to construct a triangular preconditioner for a single CholQR pass.
It achieves
$\mathcal O(u\,\cond(\mtx M))$ loss of orthogonality up to the
$\cond(\mtx M)\lesssim u^{-1}$ scale, and it exhibits greater numerical robustness than Q-less CholQR2 and sCholQR3, while using fewer operator
applications. We provide
a general, rigorous numerical analysis that accounts for
the structure of errors in forward and adjoint operator
applications. 
Synthetic experiments and least-squares benchmarks support
the analysis and demonstrate the benefits for downstream
computations. More broadly, our results advance the understanding
of numerical stability in matrix-free linear algebra and highlight
randomized sketching as a tool for numerical stabilization.

\section{Acknowledgments}
We acknowledge DARPA, NSF, the DOE Competitive Portfolios grant, and the DOE SciGPT grant for partial
support of this work.

The RandLAPACK experiments used the ISAAC Next Generation cluster at the University of Tennessee, Knoxville~\cite{ISAACNG}.

\paragraph{Declaration on the use of artificial intelligence.}
OpenAI's Codex assisted developing and reviewing theoretical results, and with designing the synthetic experiments.
Anthropic's Claude models assisted with planning, implementing, and testing
the RandLAPACK benchmarks.
The authors reviewed all AI-assisted content and tested the code.

\bibliographystyle{abbrv}
\bibliography{references/references}

\appendix
\section{Implementation details}
\label{sec:impl}

The Q-less CholQR algorithms developed in this work are implemented in
RandLAPACK, with sketching provided by RandBLAS~\cite{RandBLAS}.
This section describes the implementation choices specific to the proposed
methods, including the blocked evaluation of Gram matrices and the
construction and application of randomized preconditioners.

\subsection{Software}

RandLAPACK and RandBLAS are open-source C++ libraries for high-performance randomized numerical linear algebra, originally developed as part of the BALLISTIC project~\cite{BALLISTIC}.
RandLAPACK provides matrix factorizations and solvers, while RandBLAS provides functionality for sketching operator generation and application.
Matrix-free computation in RandLAPACK is supported through its \textit{linear operator} interface: an operator is any C++ type that implements the forward and adjoint applications of \cref{eq:linop-call}, and every Q-less QR method of \cref{subsec:ls_benchmarks} runs unchanged on any such type.
Both libraries dispatch dense kernels through BLAS\texttt{++} and LAPACK\texttt{++}~\cite{gates2022portable}, which serve as portability layers between vendor-optimized BLAS and LAPACK (such as Intel MKL~\cite{IntelMKL}, AMD AOCL~\cite{AMDAOCL}, OpenBLAS~\cite{OpenBLAS}, etc.) and C\texttt{++}.
RandBLAS uses counter-based random-number generators sourced from Random123~\cite{SalmonMoraesDrorShaw_Random123}.
The Q-less QR implementations and the benchmark drivers used below are distributed with RandLAPACK.

\subsection{Implementation basics}
\label{subsec:impl_basics}

A linear operator in RandLAPACK represents a matrix
$\mtx M\in\R^{m\times n}$ through its action on other matrices, without
requiring $\mtx M$ to be formed or stored explicitly, through
\code{GEMM}-style operations of the form
\begin{equation}
\label{eq:linop-call}
\mtx C
\leftarrow
\alpha\,\mtx M [\mtx B]
+
\beta\,\mtx C.
\end{equation}

The implementations rely primarily on Level-3 BLAS and LAPACK
kernels~\cite{BlackfordEtAl2002_BLAS,AndersonEtAl1999_LAPACK}.
The kernels relevant to this work are \code{GEMM} and \code{SYRK} for
dense matrix products and Gram-matrix computations, \code{TRSM} for
applications of triangular inverse preconditioners, \code{GEQRF} for
the QR factorization of randomized sketches, and \code{POTRF} for
Cholesky factorization. Sparse sketching operations are provided by
RandBLAS. In this work, we omit the precision prefix from BLAS
and LAPACK routine names; for example, \code{TRSM} denotes
\code{DTRSM} in double precision.

\subsection{Implementation of Q-less Cholesky QR}
\label{subsec:blocked_gram}

The linear-operator abstraction allows the Q-less QR algorithms to
operate directly on matrix-free representations of $\mtx M$.
The central computational primitive in these algorithms is the Gram
matrix
\begin{equation}
\label{eq:gram}
\mtx G
=
\mtx M^\trans[\mtx M[\mtx I]]
\in\R^{n\times n}.
\end{equation}
For an explicitly stored matrix $\mtx M$, the product
$\mtx M^\trans\mtx M$ can be computed directly using the BLAS-3
routine \code{SYRK}. For a linear operator, we instead use the
available realizations of $\mtx M[\cdot]$ and
$\mtx M^\trans[\cdot]$.

To avoid storing the full intermediate
$\mtx M[\mtx I]\in\R^{m\times n}$, we compute $\mtx G$ in column
blocks. For a block of indices $\vct j$ of size at most $b$, we compute
\[
\mtx W=\mtx M[\mtx I(:,\vct j)],
\qquad
\mtx G(:,\vct j)=\mtx M^\trans[\mtx W],
\]
and repeat over all columns of $\mtx G$. This limits the peak memory
overhead to $O(mb)$. In the least-squares benchmarks of
\cref{subsec:ls_benchmarks} we use $b=256$.

In \code{PCholQR} we define the inverse-preconditioner operator
$\mtx P^{\rm inv}[\cdot]$, $\mtx P^{\rm inv\trans}[\cdot]$ through the solution of linear systems with $\mtx P$ and $\mtx P^\trans$
with \code{TRSM}.

\paragraph{Shift option.}
In the RandLAPACK benchmarks of \cref{subsec:ls_benchmarks},
we enable shift-and-retry for CholQR, CholQR2, CQRRTO and the second
and third passes of sCholQR3. If \code{POTRF} fails, we retry
Cholesky on
\[
\mtx G+s_k\mtx I,
\]
where $s_1= u \mathrm{trace}(\mtx G)$ and $s_k=10s_{k-1}$ for $k\ge2$, until the
factorization succeeds.
CQRRTO required no retries in any experiment reported here.

\subsection{CQRRTO implementation}
\label{subsec:cqrrt-implementation}

 We next describe the implementation of the two phases of \cref{alg:cqrrt}.

We use a SASO $\mtx S\in\R^{d\times m}$ with $\zeta$ nonzeros per column, as
described in \cref{sec:precond_via_sk}. In our
experiments, $d=2n$ already provides an effective preconditioner, with
$\zeta=4$ in the least-squares benchmarks of \cref{subsec:ls_benchmarks} and
$\zeta=8$ in the synthetic study of \cref{subsec:synthetic}. The sketch is assembled column-blockwise
\[
\mtx W = \mtx M[\mtx I(:,\vct j)], \quad  \mtx M^{\mathrm{sk}}(:,\vct j)
=
\mtx S\mtx W,
\]
where $ \vct j$ is a block of indices of length at most $b$. The products $\mtx S\mtx W$ are computed with a RandBLAS~\cite{RandBLAS} sparse sketching kernel and in total require
$\zeta m n$ flops and
$\mathcal O(m(b+\zeta)+dn)$ peak storage. The QR factorization 
of the resulting $d\times n$ sketch for computing the
preconditioner $\mtx P=\mtx R^{\mathrm{sk}}$ is realized with \code{GEQRF} and costs $\mathcal O(dn^2)$ flops.

Then CQRRTO applies Q-less CholQR to the preconditioned operator $\mtx M^{\mathrm{pre}}[\cdot]
=
\mtx M[\mtx P^{\rm inv}[\cdot]],$ with $\mtx P=\mtx R^{\mathrm{sk}}$.
Its Gram matrix is computed using the implementation from~\cref{subsec:blocked_gram}, requiring
$2nT_{\mathrm{mv}}+\mathcal O(n^3)$ flops and
$\mathcal O(mb+n^2)$ peak storage. The $\mathcal O(n^3)$ term accounts
for applications of the inverse preconditioner.

The Gram matrix is factorized by the Cholesky routine \code{POTRF}
at a cost of $n^3/3$ flops. The analysis in \cref{sec:theory}
assumes that $\mtx G$ is explicitly symmetrized beforehand.
In the RandLAPACK runs of \cref{subsec:ls_benchmarks},
\code{POTRF} reads one triangle of the computed $\mtx G$
and implicitly treats the other as its transpose.
Finally, the triangular factors are accumulated as
$\mtx R=\mtx R^{\mathrm{chol}}\mtx P$ using \code{TRMM}.

\subsection{List of methods}
\label{subsec:compared_algs}
This subsection lists the algorithms used by the least-squares benchmarks of
\cref{subsec:ls_benchmarks}. 
Each Q-less QR method returns only its triangular factor
$\mtx R$, which the downstream solver uses as a right preconditioner. Every
method that uses column blocking uses $b=256$, and every sketching-based method
uses a sparse embedding with $d=2n$ rows and $\zeta=4$ nonzeros per column.

\begin{itemize}
\item \textbf{CQRRTO} (\cref{alg:cqrrt}; implementation in
\cref{subsec:cqrrt-implementation}; analysis in \cref{subsec:stab_CQRRTO}).
One preconditioned CholQR pass with $\mtx P=\mtx R^{\mathrm{sk}}$ from the
sketch, the inverse preconditioner applied through triangular solves as in
\cref{alg:pcholqr}. No shift is used.
\item \textbf{CholQR} (\cref{alg:cholqr}; analysis in
\cref{subsec:main_thm}). A single unshifted pass.
\item \textbf{CholQR2} (\cref{alg:pcholqr}; analysis in
\cref{subsec:stab_cholqr2}). Two unshifted passes, the first-pass factor
preconditioning the second.
\item \textbf{sCholQR3} (\cref{alg:scholqr3}; analysis in
\cref{subsec:stab_scholqr3}). A first pass shifted by
$s = 11\,u\,n\,\|\mtx M\|_{\mathrm F}^2$, followed by two unshifted
preconditioned passes.
\item \textbf{Blendenpik}~\cite{Blendenpik2010}. A sketch-and-precondition
least-squares solver, included as an independent baseline; it is not a Q-less QR
method. It uses the same sparse embedding as CQRRTO and takes the Householder QR
factor of the sketch as the preconditioner.  
\item \textbf{unpreconditioned}. The iterative solver applied to the raw
operator, as a reference point without factorization. Its iterations omit the
two triangular solves with $\mtx R$ and are therefore cheaper than those of the
preconditioned methods.
\end{itemize}

\section{Benchmark configuration and additional results}\label{sec:ls_set}

\subsection{Hardware and software setup}\label{subsec:exp_set}
\paragraph{Platform.}
All RandLAPACK benchmark runs were executed on the ISAAC Next Generation (ISAAC-NG) cluster~\cite{ISAACNG} of the University of Tennessee, Knoxville, on a dual-socket Intel Xeon Gold 6430 (Sapphire Rapids) node from its large-memory partition; \cref{table:cpu_config} lists the configuration.
Every run reported here was executed on the same physical node of that partition, allocated all $64$ of its cores with hyper-threading disabled, so wall-clock times are comparable across problems and methods.
The software stack is the one described in \cref{sec:impl}; \cref{table:cpu_config} records the versions.

\paragraph{Threading and timing protocol.}
Each run uses 64 OpenMP threads, one per physical core, set through \code{OMP\_NUM\_THREADS} and inherited by MKL, and is launched with \code{numactl --interleave=all} so that memory is spread evenly over the two sockets.
The triangular solves of the inner iteration (BLAS-2) and the FFTs are capped at 16 threads, and the triangular solves at 8 threads when $n \le 4000$; these caps were the fastest settings measured on this hardware.
A benchmark run starts with an untimed warm-up at the problem dimensions (a Q-less factorization, a few iterations of the solver, and a sketch with its QR factorization), which initializes the thread pools and touches all large allocations.
Each method is then executed three times in a row.
The reported run of a method is the one with the smallest total wall-clock time (factorization plus solve), and every plotted quantity of that method, including accuracy metrics and iteration counts, comes from that single run.
The sketching-based methods draw a fresh sketch in each run; all other methods are deterministic and reproduce their accuracy metrics and iteration counts bit for bit across the three runs.
Each run records the host name, CPU model, thread count, environment settings, and RandLAPACK commit alongside its results.

\begin{table}[H]
\small\def\arraystretch{1.3}
\caption[Hardware and software configuration of the benchmark node.]{\small Hardware and software configuration of the benchmark node (ISAAC-NG, University of Tennessee, Knoxville).}
\label{table:cpu_config}
\centering
\begin{adjustbox}{max width=\textwidth}
\begin{tabular}{|cc||c|}
\hline
\multicolumn{2}{|c||}{\textbf{}} & \textbf{Intel Xeon Gold 6430 (2$\times$, Sapphire Rapids)} \\ \hline \hline
\multicolumn{2}{|c||}{\textbf{Cores per socket}} & 32 (hyper-threading disabled) \\ \hline
\multicolumn{2}{|c||}{\textbf{Threads used in experiments}} & 64 (one per core, all cores of the node) \\ \hline
\multicolumn{1}{|c|}{\multirow{2}{*}{\textbf{Clock speed}}} & \textbf{Base} & 2.10 GHz \\ \cline{2-3}
\multicolumn{1}{|c|}{} & \textbf{Max turbo} & 3.40 GHz \\ \hline
\multicolumn{1}{|c|}{\multirow{3}{*}{\textbf{Cache sizes}}} & \textbf{L1 (per core)} & 80 KB (32 KB instruction + 48 KB data) \\ \cline{2-3}
\multicolumn{1}{|c|}{} & \textbf{L2 (per core)} & 2 MB \\ \cline{2-3}
\multicolumn{1}{|c|}{} & \textbf{L3 (per socket, shared)} & 60 MB \\ \hline
\multicolumn{2}{|c||}{\textbf{RAM}} & 2016 GB \\ \hline
\multicolumn{2}{|c||}{\textbf{FP64 peak performance}} & 2.15 TFLOP/s per socket (theoretical, base clock) \\ \hline
\multicolumn{2}{|c||}{\textbf{BLAS, LAPACK, FFT}} & Intel oneMKL 2025.3 (ILP64, GNU threading layer) \\ \hline
\multicolumn{2}{|c||}{\textbf{Compiler}} & GCC 13.4.0, \code{-O3} (CMake \code{Release}) \\ \hline
\multicolumn{2}{|c||}{\textbf{OpenMP runtime}} & GNU libgomp (GCC 13.4.0) \\ \hline
\multicolumn{2}{|c||}{\textbf{CMake}} & 3.31.8 \\ \hline
\multicolumn{2}{|c||}{\textbf{RandLAPACK}} & commit \code{ebe383f}, including the benchmark drivers \\ \hline
\multicolumn{2}{|c||}{\textbf{OS}} & Red Hat Enterprise Linux 8.10, kernel 4.18.0-553 \\ \hline
\end{tabular}
\end{adjustbox}
\end{table}

\subsection{Additional benchmark cases}\label{subsec:extra_cases}

\Cref{subsec:results_fem2,subsec:results_toeplitz} report one representative
problem per benchmark. The remaining cases from the same campaign are collected
here. Each figure shows, from left to right, wall-clock time stacked as in
\cref{fig:benchmarks}, inner CG iterations totaled over all refinement cycles,
orthogonality loss, and the error the stopping rule targets (data error for FEM,
backward error for Toeplitz, its target $\sqrt n\,u$ dashed), all defined as in
\cref{tab:fem-observations,tab:toeplitz-observations}. A lighter bar with a dashed outline marks a method
that stopped without reaching the error target. The unpreconditioned solver forms no
preconditioner, so it has no orthogonality-loss bar. In every case, CholQR and CholQR2 needed one
Cholesky shift retry each, sCholQR3 needed one only on the $301864\times33024$ FEM operator, and CQRRTO needed none.

\paragraph{FEM benchmark.}
\Cref{fig:fem2_small,fig:fem2_large} report results for the additional advection–diffusion–reaction FEM operators. CholQR and the unpreconditioned solver stop above the data-error floor.

\begin{figure}[H]
    \centering
    \includegraphics[width=\textwidth]{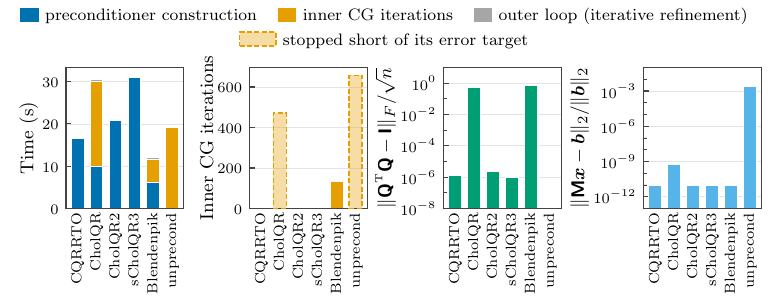}
    \caption{FEM benchmark for a $75466\times8256$ operator with
    $\cond(\mtx M)\approx5\times10^{11}$. }
    \label{fig:fem2_small}
\end{figure}

\begin{figure}[H]
    \centering
    \includegraphics[width=\textwidth]{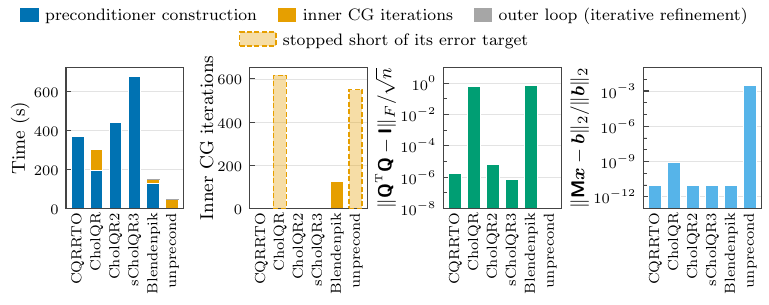}
    \caption{FEM benchmark for a $301864\times33024$ operator with
     $\cond(\mtx M)\approx5\times10^{11}$.}
    \label{fig:fem2_large}
\end{figure}

\paragraph{Toeplitz benchmark.}
\Cref{fig:toeplitz_16000,fig:toeplitz_8256,fig:toeplitz_301864} cover experiments for the additional Toeplitz operators of different sizes. Only the unpreconditioned solver stops short of the backward-error target.

\begin{figure}[H]
    \centering
    \includegraphics[width=\textwidth]{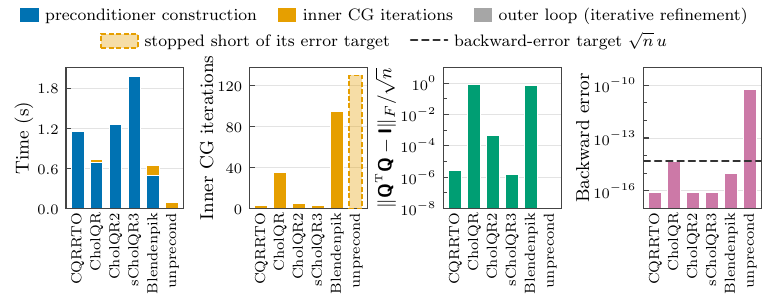}
    \caption{Toeplitz benchmark for a $16000\times2000$ operator. }
    \label{fig:toeplitz_16000}
\end{figure}

\begin{figure}[H]
    \centering
    \includegraphics[width=\textwidth]{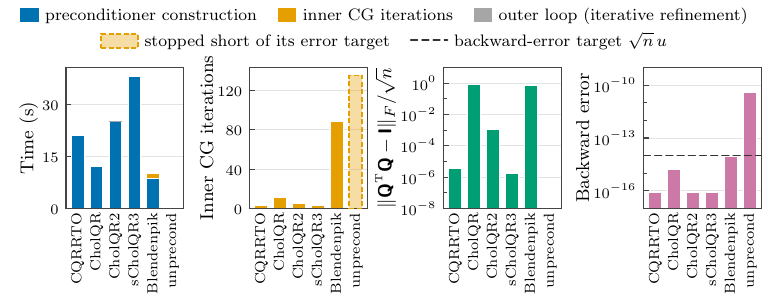}
    \caption{Toeplitz benchmark for a $75466\times8256$ operator. }
    \label{fig:toeplitz_8256}
\end{figure}

\begin{figure}[H]
    \centering
    \includegraphics[width=\textwidth]{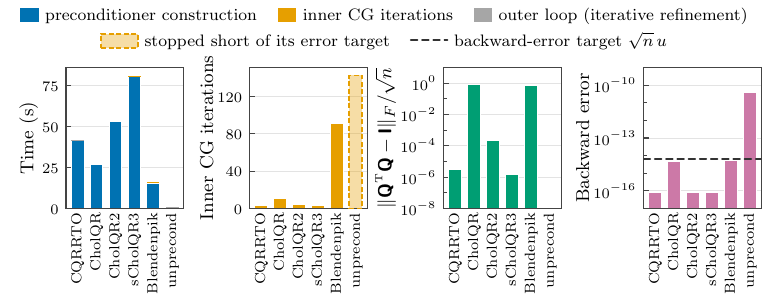}
    \caption{Toeplitz benchmark for a $301864\times3019$ operator.}
    \label{fig:toeplitz_301864}
\end{figure}

\section{Extended Results and Proofs}\label{app:proofs}

This appendix collects the full theorem statements and proofs of the results in \cref{sec:theory}.
Throughout the preconditioned and multipass proofs, inverse actions
satisfy Assumption~\ref{as:inverse-preconditioner}, and products used to
accumulate triangular factors are interpreted in exact arithmetic,
as in \cref{thm:pcholqr-stability,rem:accumulated-product}.
The shifted-pass analysis likewise treats $\widehat{\mtx G}+s\mtx I$
as the exact matrix supplied to Cholesky.

\subsection{Forward error of preconditioned operator}
\label{app:preconditioned-forward-proof}

\begin{proof}[Proof of Proposition~\ref{prop:preconditioned-forward-shapers}]
Write
$\mtx P^{\rm inv}[\mtx X]
=\mtx P^{-1}(\mtx X+\mtx E_P)$.
Applying \cref{as:inverse-preconditioner} columnwise and using
$\gamma_n\cond_F(\mtx P)\le 1/2$ gives
\begin{align*}
\|\mtx E_P\|_F
&\le
\frac{\gamma_n}{1-\gamma_n\cond_F(\mtx P)}
\max_{\|\vct x\|_2=1}
\bigl\|\,|\mtx P|\,|\mtx P^{-1}\vct x|\,\bigr\|_2
\|\mtx X\|_F
\\
&\le
2\gamma_n
\min\left\{
\cond_F(\mtx P),
\bigl\|\,|\mtx P|\,|\mtx P^{-1}|\,\bigr\|_2
\right\}
\|\mtx X\|_F.
\end{align*}
By \cref{eq:comp-matvec-M},
\begin{align*}
\mtx M^{\rm pre}[\mtx X]-\mtx M^{\rm pre}\mtx X
&=
\mtx M^{\rm pre}\mtx E_P
+\mtx\Delta_{\rm out}^{M}\mtx E_M,
\\
\|\mtx E_M\|_F
&\le
u\|\mtx\Delta_{\rm in}^{M}
\mtx P^{-1}(\mtx X+\mtx E_P)\|_F
\\
&\le
u\|\mtx\Delta_{\rm in}^{M}\mtx P^{-1}\mtx X\|_F
+\mathcal O(u^2)\|\mtx X\|_F,
\end{align*}
so stacking the two error bounds gives the stated shapers.
\end{proof}
\subsection{Stability of Q-less CholQR}\label{app:proofs_cholqr}

\begin{theorem}[Stability of Q-less CholQR]
\label{thm:cholqr-stability}
Let $\mtx M\in\R^{m\times n}$ have full column rank and suppose that it
satisfies Definition~\ref{def:operator-model}. Let
\[
\mtx M=\mtx Q\mtx R
\]
denote its exact QR factorization, with $\mtx R$ having positive
diagonal. Define
\begin{equation}
\label{eq:cholqr-operator-error-full}
\begin{aligned}
u_M
&:={}
u
\frac{\|\mtx Q^\trans\mtx\Delta_{\rm out}^{M}\|_2}
     {\|\mtx M\|_F}
\|\mtx\Delta_{\rm in}^{M}\|_F
+
u
\|\mtx R^{-\trans}\mtx\Delta_{\rm out}^{M^*}\|_2\|\mtx\Delta_{\rm in}^{M^*}\mtx Q\|_2\\
&+ u^2
\|\mtx R^{-\trans}\mtx\Delta_{\rm out}^{M^*}\|_2
\|\mtx\Delta_{\rm in}^{M^*}\|_2
\|\mtx\Delta_{\rm in}^{M}\|_F
\frac{
\|\mtx\Delta_{\rm out}^{M}\|_2
}{
\|\mtx M\|_F
},
\end{aligned}
\end{equation}
Consider \cref{alg:cholqr} and set
\[
\widehat{\mtx Q}
:=
\mtx M\widehat{\mtx R}^{-1}.
\]
If, in addition to the roundoff regime described by
Assumption~\ref{as:roundoff},
\begin{equation}
\label{eq:cholqr-smallness-full}
1.05\,u_M\cond_F(\mtx M)
+
1.5\,\gamma_{n+1}\cond_F(\mtx M)^2
\le0.2, \quad 
1.6\,n
\gamma_{n+1}\cond_F(\mtx M)^2
\le1,
\end{equation}
then the Cholesky factorization succeeds, $\widehat{\mtx R}$ is nonsingular and
\begin{equation}
\label{eq:cholqr-Q-bound-full}
\boxed{
\left\|
\widehat{\mtx Q}^{\trans}\widehat{\mtx Q}-\mtx I
\right\|_F
\le
1.25\,u_M\cond_F(\mtx M)
+
1.75\,\gamma_{n+1}\cond_F(\mtx M)^2 .
}
\end{equation}
\end{theorem}
\begin{proof}[Proof of \cref{thm:cholqr-stability}]
Write $\operatorname{sym}(\mtx B)=(\mtx B+\mtx B^\trans)/2$.
By Definition~\ref{def:operator-model}, the two operator applications
satisfy
\[
\widetilde{\mtx M}=\mtx M+\mtx\Delta_{\rm out}^{M}\mtx E,
\qquad
\mtx M^\trans[\widetilde{\mtx M}]
=\mtx M^\trans\widetilde{\mtx M}
 +\mtx\Delta_{\rm out}^{M^*}\mtx E^*,
\]
where
\begin{align*}
\|\mtx E\|_F
&\le u\|\mtx\Delta_{\rm in}^{M}\|_F,
\\
\|\mtx E^*\|_F
&\le u\|\mtx\Delta_{\rm in}^{M^*}\mtx Q\|_2\|\mtx M\|_F
 +u^2\|\mtx\Delta_{\rm in}^{M^*}\|_2
       \|\mtx\Delta_{\rm out}^{M}\|_2
       \|\mtx\Delta_{\rm in}^{M}\|_F.
\end{align*}
Thus the symmetrized computed Gram matrix is
\[
\widehat{\mtx G}=\mtx M^\trans\mtx M+\mtx F,
\qquad
\mtx F=\operatorname{sym}\!\left(
 \mtx M^\trans\mtx\Delta_{\rm out}^{M}\mtx E
 +\mtx\Delta_{\rm out}^{M^*}\mtx E^*\right).
\]
Using $\mtx R^{-\trans}\operatorname{sym}(\mtx B)\mtx R^{-1}
=\operatorname{sym}(\mtx R^{-\trans}\mtx B\mtx R^{-1})$,
the fact that symmetrization does not increase the Frobenius norm,
and $\mtx R^{-\trans}\mtx M^\trans=\mtx Q^\trans$, we obtain
\begin{equation}
\label{eq:cholqr-normalized-Gram}
\begin{aligned}
\|\mtx R^{-\trans}\mtx F\mtx R^{-1}\|_F
&\le \|\mtx R^{-1}\|_2
 \left(
 \|\mtx Q^\trans\mtx\Delta_{\rm out}^{M}\|_2\|\mtx E\|_F
 +\|\mtx R^{-\trans}\mtx\Delta_{\rm out}^{M^*}\|_2
   \|\mtx E^*\|_F\right)
\\
&\le u_M\cond_F(\mtx M)=:\rho.
\end{aligned}
\end{equation}

The first inequality in \cref{eq:cholqr-smallness-full} gives
$\rho\le4/21$. Hence $\widehat{\mtx G}$ is positive definite and
\[
\cond_2(\widehat{\mtx G})
\le\frac{1+\rho}{1-\rho}\cond_2(\mtx M)^2
\le\frac{25}{17}\cond_F(\mtx M)^2.
\]
Together with the second inequality in
\cref{eq:cholqr-smallness-full}, this implies
\[
1.01\,n\gamma_{n+1}\cond_2(\widehat{\mtx G})
\le1.01\frac{25}{17}\,n\gamma_{n+1}\cond_F(\mtx M)^2<1.
\]
Assumption~\ref{as:roundoff} therefore guarantees success and the
backward-error relation
\begin{equation}
\label{eq:cholqr-backward}
\widehat{\mtx R}^\trans\widehat{\mtx R}
=\widehat{\mtx G}+\mtx E^{\rm chol},
\qquad
|\mtx E^{\rm chol}|
\le\gamma_{n+1}|\widehat{\mtx R}|^\trans|\widehat{\mtx R}|.
\end{equation}
Taking traces and using
$|\trace(\mtx E^{\rm chol})|
 \le\gamma_{n+1}\|\widehat{\mtx R}\|_F^2$ yields
\begin{equation}
\label{eq:cholqr-Rhat-bound}
\|\widehat{\mtx R}\|_F^2
\le\frac{\trace(\widehat{\mtx G})}{1-\gamma_{n+1}}
\le\frac{1+\rho}{1-\gamma_{n+1}}\|\mtx M\|_F^2.
\end{equation}
Consequently,
\begin{equation}
\label{eq:cholqr-normalized-chol}
\|\mtx R^{-\trans}\mtx E^{\rm chol}\mtx R^{-1}\|_F
\le\frac{1+\rho}{1-\gamma_{n+1}}
       \gamma_{n+1}\cond_F(\mtx M)^2
\le1.4\,\gamma_{n+1}\cond_F(\mtx M)^2,
\end{equation}
where the last bound follows from $\rho\le4/21$ and
$\gamma_{n+1}\le2/15$, both consequences of
\cref{eq:cholqr-smallness-full}.

Set
\[
\mtx H=\mtx R^{-\trans}(\mtx F+\mtx E^{\rm chol})\mtx R^{-1}.
\]
By \cref{eq:cholqr-normalized-Gram,eq:cholqr-normalized-chol},
\[
\|\mtx H\|_F
\le\rho+1.4\,\gamma_{n+1}\cond_F(\mtx M)^2\le0.2,
\qquad
\widehat{\mtx R}^\trans\widehat{\mtx R}
=\mtx R^\trans(\mtx I+\mtx H)\mtx R.
\]
In particular, $\widehat{\mtx R}$ is nonsingular and
$\|\mtx R\widehat{\mtx R}^{-1}\|_2^2
 \le(1-\|\mtx H\|_2)^{-1}$.
The identity
\[
\widehat{\mtx Q}^\trans\widehat{\mtx Q}-\mtx I
=-(\mtx R\widehat{\mtx R}^{-1})^\trans
       \mtx H(\mtx R\widehat{\mtx R}^{-1})
\]
therefore gives
\[
\|\widehat{\mtx Q}^\trans\widehat{\mtx Q}-\mtx I\|_F
\le\frac{\|\mtx H\|_F}{1-\|\mtx H\|_2}
\le1.25\,u_M\cond_F(\mtx M)
  +1.75\,\gamma_{n+1}\cond_F(\mtx M)^2,
\]
as required.
\end{proof}

\subsection{Stability of preconditioned Q-less CholQR}\label{app:proofs_pcholqr}

\begin{theorem}[Stability of preconditioned Q-less CholQR]
\label{thm:pcholqr-stability}
Let $\mtx M\in\R^{m\times n}$ have full column rank and suppose that it
satisfies Definition~\ref{def:operator-model}. Let
$\mtx P\in\R^{n\times n}$ be a nonsingular upper-triangular
preconditioner satisfying Assumption~\ref{as:inverse-preconditioner}, and suppose that
\[
\gamma_n\cond_F(\mtx P)\le\frac12.
\]
Consider \cref{alg:pcholqr} and define
\[
\mtx M^{\rm pre}:=\mtx M\mtx P^{-1},
\qquad
\widehat{\mtx Q}
:=
\mtx M\widehat{\mtx R}^{-1},
\qquad
\widehat{\mtx R}
=
\widehat{\mtx R}^{\rm chol}\mtx P,
\]
where the matrix products are interpreted in exact arithmetic.
Let
\[
\mtx M=\mtx Q\mtx R
\]
be the exact QR factorization of $\mtx M$.

Define
\begin{equation}
\label{eq:pcholqr-upre-full}
\begin{aligned}
u_{\rm pre}
:={}&
2.5\,\gamma_n\cond_F(\mtx P)
\left(
\sqrt n+\cond_2(\mtx M^{\rm pre})
\right)
\\
&+
1.25\,u
\|\mtx Q^\trans\mtx\Delta_{\rm out}^{M}\|_2
\|\mtx\Delta_{\rm in}^{M}\mtx R^{-1}\|_2
\\
&+
1.25\,u
\|\mtx R^{-\trans}\mtx\Delta_{\rm out}^{M^*}\|_2
\|\mtx\Delta_{\rm in}^{M^*}\mtx Q\|_2
\\
&+
2.5\,u\gamma_n\sqrt n\,\cond_F(\mtx P)
\|\mtx R^{-\trans}\mtx\Delta_{\rm out}^{M^*}\|_2
\|\mtx\Delta_{\rm in}^{M^*}\|_2
\\
&+
1.5\,u^2
\|\mtx R^{-\trans}\mtx\Delta_{\rm out}^{M^*}\|_2
\|\mtx\Delta_{\rm in}^{M^*}\|_2
\|\mtx\Delta_{\rm in}^{M}\mtx R^{-1}\|_2
\|\mtx\Delta_{\rm out}^{M}\|_2.
\end{aligned}
\end{equation}

If, in addition to the roundoff regime of Assumption~\ref{as:roundoff},
\begin{equation}
\label{eq:pcholqr-smallness-full}
1.05\,u_{\rm pre}
\cond_F(\mtx M^{\rm pre})
+
1.5\,\gamma_{n+1}
\cond_F(\mtx M^{\rm pre})^2
\le 0.2,
\qquad
1.6\,n\gamma_{n+1}
\cond_F(\mtx M^{\rm pre})^2
\le 1,
\end{equation}
then the Cholesky factorization succeeds,
$\widehat{\mtx R}^{\rm chol}$ is nonsingular, and
\begin{equation}
\label{eq:pcholqr-Q-bound-full}
\boxed{
\left\|
\widehat{\mtx Q}^{\trans}\widehat{\mtx Q}-\mtx I
\right\|_F
\le
1.25\,u_{\rm pre}
\cond_F(\mtx M^{\rm pre})
+
1.75\,\gamma_{n+1}
\cond_F(\mtx M^{\rm pre})^2 .
}
\end{equation}
\end{theorem}

\begin{proof}[Proof of \cref{thm:pcholqr-stability}]
Set $\mtx A=\mtx M^{\rm pre}$ and
$\mtx R^{\rm pre}=\mtx R\mtx P^{-1}$, so that
$\mtx A=\mtx Q\mtx R^{\rm pre}$. Write
\[
a=2\gamma_n\sqrt n\,\cond_F(\mtx P),
\qquad
b=2\gamma_n\cond_F(\mtx P)\cond_2(\mtx A).
\]
Since $u_{\rm pre}\ge1.25(a+b)$, the first condition in
\cref{eq:pcholqr-smallness-full} implies
\begin{equation}
\label{eq:pcholqr-solve-smallness}
a+b\le\frac{16}{105},
\qquad
a+b+ab\le1.25(a+b),
\qquad
(1+a)(1+b)<1.25.
\end{equation}

Assumption~\ref{as:inverse-preconditioner} and
$\gamma_n\cond_F(\mtx P)\le1/2$ give the residual representation
\begin{equation}
\label{eq:pcholqr-forward-solve-error}
\mtx P^{\rm inv}[\mtx I]=\mtx P^{-1}(\mtx I+\mtx E_P),
\qquad
\|\mtx E_P\|_F
\le2\gamma_n\|\mtx P\|_F\|\mtx P^{-1}\|_F\le a.
\end{equation}
The subsequent application of $\mtx M$ satisfies
\[
\widetilde{\mtx A}
=\mtx M[\mtx P^{\rm inv}[\mtx I]]
=\mtx A+\mtx F_{\rm f},
\qquad
\mtx F_{\rm f}=\mtx A\mtx E_P+\mtx\Delta_{\rm out}^{M}\mtx E_M,
\]
where, using $\mtx P^{-1}=\mtx R^{-1}\mtx R^{\rm pre}$,
\[
\|\mtx E_M\|_F
\le u(1+a)\|\mtx\Delta_{\rm in}^{M}\mtx R^{-1}\|_2
             \|\mtx A\|_F.
\]
For the adjoint application, write
\[
\mtx Z=\mtx M^\trans[\widetilde{\mtx A}]
=\mtx M^\trans\widetilde{\mtx A}
 +\mtx\Delta_{\rm out}^{M^*}\mtx E_{M^*}.
\]
Definition~\ref{def:operator-model} gives
\[
\|\mtx E_{M^*}\|_F
\le u\|\mtx\Delta_{\rm in}^{M^*}\mtx Q\|_2\|\mtx A\|_F
 +u\|\mtx\Delta_{\rm in}^{M^*}\|_2\|\mtx F_{\rm f}\|_F.
\]
These bounds and the expression for $\mtx F_{\rm f}$ imply
\begin{align*}
&\frac{
 \|\mtx Q^\trans\mtx F_{\rm f}\|_F
 +\|\mtx R^{-\trans}\mtx\Delta_{\rm out}^{M^*}\|_2
  \|\mtx E_{M^*}\|_F}{\|\mtx A\|_F}
\\
&\quad\le a+(1+a)\Bigl[
 u\|\mtx Q^\trans\mtx\Delta_{\rm out}^{M}\|_2
   \|\mtx\Delta_{\rm in}^{M}\mtx R^{-1}\|_2
\\
&\qquad\quad+
 u\|\mtx R^{-\trans}\mtx\Delta_{\rm out}^{M^*}\|_2
   \|\mtx\Delta_{\rm in}^{M^*}\mtx Q\|_2
 +ua\|\mtx R^{-\trans}\mtx\Delta_{\rm out}^{M^*}\|_2
     \|\mtx\Delta_{\rm in}^{M^*}\|_2
\\
&\qquad\quad+
 u^2\|\mtx R^{-\trans}\mtx\Delta_{\rm out}^{M^*}\|_2
     \|\mtx\Delta_{\rm in}^{M^*}\|_2
     \|\mtx\Delta_{\rm in}^{M}\mtx R^{-1}\|_2
     \|\mtx\Delta_{\rm out}^{M}\|_2\Bigr].
\end{align*}
The bracket is at most $\frac45u_{\rm pre}-a-b$ by
\cref{eq:pcholqr-upre-full}. In particular, including all mixed and
higher-order terms,
\begin{equation}
\label{eq:pcholqr-forward-adjoint-bound}
\frac{
 \|\mtx Q^\trans\mtx F_{\rm f}\|_F
 +\|\mtx R^{-\trans}\mtx\Delta_{\rm out}^{M^*}\|_2
  \|\mtx E_{M^*}\|_F}{\|\mtx A\|_F}
\le a+(1+a)\left(\frac45u_{\rm pre}-a-b\right).
\end{equation}

The final adjoint inverse application has the analogous residual
representation
\begin{equation}
\label{eq:pcholqr-adjoint-solve-error}
\mtx P^{\rm invT}[\mtx Z]=\mtx P^{-\trans}(\mtx Z+\mtx E_P^*),
\qquad
\|\mtx E_P^*\|_F
\le2\gamma_n\|\mtx P\|_F\|\mtx P^{-\trans}\mtx Z\|_F.
\end{equation}
Using $\mtx P^{-\trans}\mtx M^\trans=\mtx A^\trans$ and
$\mtx P^{-\trans}\mtx\Delta_{\rm out}^{M^*}
 =(\mtx R^{\rm pre})^\trans
   \mtx R^{-\trans}\mtx\Delta_{\rm out}^{M^*}$ gives
\begin{equation}
\label{eq:pcholqr-adjoint-solve-bound}
\begin{aligned}
\|\mtx R^{-\trans}\|_2\|\mtx E_P^*\|_F
\le b\Bigl(&\|\mtx A\|_F+\|\mtx Q^\trans\mtx F_{\rm f}\|_F
\\
&+\|\mtx R^{-\trans}\mtx\Delta_{\rm out}^{M^*}\|_2
   \|\mtx E_{M^*}\|_F\Bigr).
\end{aligned}
\end{equation}
Here we also used
$\|\mtx R^{-\trans}\|_2
 \le\|(\mtx R^{\rm pre})^{-1}\|_2\|\mtx P^{-1}\|_2$.

The symmetrized matrix supplied to Cholesky is
$\widehat{\mtx G}=\mtx A^\trans\mtx A+\mtx F$, where
\[
\mtx F=\operatorname{sym}\!\left(
\mtx A^\trans\mtx F_{\rm f}
+\mtx P^{-\trans}\mtx\Delta_{\rm out}^{M^*}\mtx E_{M^*}
+\mtx P^{-\trans}\mtx E_P^*\right).
\]
Since
$(\mtx R^{\rm pre})^{-\trans}\mtx A^\trans=\mtx Q^\trans$ and
$(\mtx R^{\rm pre})^{-\trans}\mtx P^{-\trans}=\mtx R^{-\trans}$,
\cref{eq:pcholqr-forward-adjoint-bound,eq:pcholqr-adjoint-solve-bound}
yield
\begin{align*}
\frac{\|(\mtx R^{\rm pre})^{-\trans}\mtx F
                (\mtx R^{\rm pre})^{-1}\|_F}{\cond_F(\mtx A)}
&\le b+(1+b)\left[a+(1+a)
                 \left(\frac45u_{\rm pre}-a-b\right)\right]
\\
&\le u_{\rm pre},
\end{align*}
where the last step uses \cref{eq:pcholqr-solve-smallness} and
$\frac45u_{\rm pre}-a-b\ge0$. Thus
\begin{equation}
\label{eq:pcholqr-normalized-Gram}
\|(\mtx R^{\rm pre})^{-\trans}\mtx F
             (\mtx R^{\rm pre})^{-1}\|_F
\le u_{\rm pre}\cond_F(\mtx M^{\rm pre}).
\end{equation}

Starting from \cref{eq:cholqr-normalized-Gram}, the Cholesky part of
the proof of \cref{thm:cholqr-stability} uses only this normalized
Gram bound. It therefore applies with $\mtx M$ replaced by
$\mtx M^{\rm pre}$ and $u_M$ by $u_{\rm pre}$.
Under \cref{eq:pcholqr-smallness-full}, the Cholesky factorization
succeeds, $\widehat{\mtx R}^{\rm chol}$ is nonsingular, and
\cref{eq:pcholqr-Q-bound-full} follows because exact accumulation gives
\[
\widehat{\mtx Q}
=\mtx M(\widehat{\mtx R}^{\rm chol}\mtx P)^{-1}
=\mtx M^{\rm pre}(\widehat{\mtx R}^{\rm chol})^{-1}.
\]
\end{proof}

\begin{proof}[Proof of Corollary~\ref{cor:pcholqr-blackbox}]
Substituting the error-shaping operators from
\cref{ex:blackbox-operator} into \cref{eq:pcholqr-upre-full} gives
\begin{align*}
u_{\rm pre}
={}&2.5\,\gamma_n\cond_F(\mtx P)
 \bigl(\sqrt n+\cond_2(\mtx M^{\rm pre})\bigr)
 +2.5\,Cu\cond_F(\mtx M)
\\
&+Cu\cond_F(\mtx M)
 \left[2.5\,\gamma_n\sqrt n\,\cond_F(\mtx P)
       +1.5\,Cu\cond_F(\mtx M)\right].
\end{align*}
The square bracket is at most $u_{\rm pre}$, and the first inequality
in \cref{eq:pcholqr-smallness-full} gives
$u_{\rm pre}\le4/21<0.31$. This proves the asserted bound for the
full coefficient. Substitution in \cref{thm:pcholqr-stability}, using
$1.25\cdot2.81<3.513$, proves the orthogonality estimate.

The same argument also gives a sufficient test used below: it suffices
to verify \cref{eq:pcholqr-smallness-full} with $u_{\rm pre}$ replaced
by the displayed black-box majorant in the corollary. Indeed, that
majorant also dominates the square bracket, so its first smallness
condition again bounds the bracket by $4/21$ and proves that it
bounds the full $u_{\rm pre}$. Moreover, it implies
$\gamma_n\cond_F(\mtx P)\le8/105<1/2$.
Thus this test verifies both the full coefficient condition and the
inverse-solve smallness condition.
\end{proof}

\subsection{Stability of Q-less CholQR2}
\label{app:proofs_cholqr2}

\begin{proof}[Proof of Corollary~\ref{cor:qless-cholqr2-blackbox}]
For the black-box shapers, \cref{eq:cholqr-operator-error-full} gives
\[
u_M=Cu\bigl(\sqrt n+\cond_F(\mtx M)\bigr)
     +C^2u^2\sqrt n\,\cond_F(\mtx M).
\]
Since $\cond_F(\mtx M)\ge\sqrt n$, the first inequality in
\cref{eq:qr2-blackbox-first-smallness} implies $Cu\sqrt n<0.03$ and
therefore
\[
u_M\le2.03\,Cu\cond_F(\mtx M)
     \le u_{\rm qr2}\cond_F(\mtx M).
\]
Thus \cref{eq:qr2-blackbox-first-smallness} verifies both smallness
conditions of \cref{thm:cholqr-stability}. The first pass succeeds
and, with $\mtx P=\widehat{\mtx R}^{\rm chol}_1$ and
$\mtx Q_1=\mtx M\mtx P^{-1}$, the theorem yields
\[
\|\mtx Q_1^\trans\mtx Q_1-\mtx I\|_F\le0.1.
\]
Consequently,
\[
\cond_2(\mtx Q_1)\le\sqrt{\frac{11}{9}}<1.11,
\qquad
\cond_F(\mtx Q_1)\le\sqrt{\frac{11n}{9}},
\qquad
\cond_F(\mtx P)\le\sqrt{\frac{11}{9}}\cond_F(\mtx M),
\]
where the last estimate follows from $\mtx M=\mtx Q_1\mtx P$.
Since $2.5\sqrt{11/9}<2.77$, the black-box majorant in
Corollary~\ref{cor:pcholqr-blackbox} is at most
$u_{\rm qr2}\cond_F(\mtx M)$.
These estimates and \cref{eq:qr2-blackbox-second-smallness} verify
the sufficient test established in the proof of that corollary.
In particular, both smallness conditions of
\cref{thm:pcholqr-stability} and its inverse-solve condition hold.

The second pass therefore succeeds, its accumulated triangular
factor is nonsingular, and \cref{thm:pcholqr-stability} gives
\[
\|\widehat{\mtx Q}^\trans\widehat{\mtx Q}-\mtx I\|_F
\le1.25\sqrt{\frac{11n}{9}}\,
      u_{\rm qr2}\cond_F(\mtx M)
  +1.75\frac{11n}{9}\gamma_{n+1}.
\]
Rounding the two numerical coefficients upward yields
\cref{eq:qr2-blackbox-final-bound}.
\end{proof}

\subsection{Stability of Q-less shifted CholQR3}
\label{app:proofs_cholqr3}

\begin{proof}[Proof of Lemma~\ref{lem:shifted-first-pass}]
It suffices to consider $u>0$, since the zero-roundoff case is exact.
Write
\[
\widetilde{\mtx M}=\mtx M[\mtx I]=\mtx M+\mtx D,
\qquad
\mtx M^\trans[\widetilde{\mtx M}]
=\mtx M^\trans\widetilde{\mtx M}+\mtx E.
\]
The black-box bounds give
\[
\|\mtx D\|_F\le Cu\sqrt n\,\|\mtx M\|_F,
\qquad
\|\mtx E\|_F\le Cu\|\mtx M\|_F\|\widetilde{\mtx M}\|_F.
\]
Thus $\widehat{\mtx G}=\mtx M^\trans\mtx M+\mtx F$, with
$\mtx F=\operatorname{sym}(\mtx M^\trans\mtx D+\mtx E)$, and
\begin{align*}
\|\mtx F\|_F
&\le Cu\bigl(\sqrt n+1+Cu\sqrt n\bigr)\|\mtx M\|_F^2
 \le2.1\,Cu\sqrt n\,\|\mtx M\|_F^2,
\\
|\trace(\mtx F)|
&\le\|\mtx M\|_F\|\mtx D\|_F+\sqrt n\,\|\mtx E\|_F
\\
&\le Cu\sqrt n\bigl(2+Cu\sqrt n\bigr)\|\mtx M\|_F^2
 \le0.2\,\|\mtx M\|_F^2.
\end{align*}
Both bounds use the first inequality in
\cref{eq:shifted-pass-smallness}.
The shift condition also gives $\|\mtx F\|_2\le s/4$, whence
\[
\lambda_{\min}(\widehat{\mtx G}+s\mtx I)\ge\frac34s,
\qquad
\lambda_{\max}(\widehat{\mtx G}+s\mtx I)
\le\|\mtx M\|_F^2+\frac54s.
\]
Since $s\ge2n\gamma_{n+1}\|\mtx M\|_F^2$ and
$n\gamma_{n+1}\le1/6$,
\[
1.01\,n\gamma_{n+1}\cond_2(\widehat{\mtx G}+s\mtx I)
\le1.01\left(\frac23+\frac5{18}\right)<1.
\]
Assumption~\ref{as:roundoff} therefore guarantees success and
\[
\widehat{\mtx R}_1^\trans\widehat{\mtx R}_1
=\mtx M^\trans\mtx M+s\mtx I+\mtx F+\mtx E^{\rm chol}.
\]
The first inequality in \cref{eq:cholqr-Rhat-bound}, now with input
$\widehat{\mtx G}+s\mtx I$, gives
\[
\|\widehat{\mtx R}_1\|_F^2
\le\frac{1.2\|\mtx M\|_F^2+ns}{1-\gamma_{n+1}},
\qquad
\|\mtx E^{\rm chol}\|_2
\le\gamma_{n+1}\|\widehat{\mtx R}_1\|_F^2.
\]
This proves the Frobenius-norm bound in \cref{eq:shift-R-bounds}.
Using \cref{eq:shifted-pass-smallness} and $\gamma_{n+1}\le1/6$,
\begin{align*}
\|\mtx F+\mtx E^{\rm chol}\|_2
&\le\bigl(2.1\,Cu\sqrt n+1.5\,\gamma_{n+1}\bigr)\|\mtx M\|_F^2
  +\frac{n\gamma_{n+1}}{1-\gamma_{n+1}}s
\\
&\le\frac14s+\frac15s=0.45\,s.
\end{align*}
Consequently,
\[
\mtx M^\trans\mtx M+0.55\,s\mtx I
\preceq\widehat{\mtx R}_1^\trans\widehat{\mtx R}_1
\preceq\mtx M^\trans\mtx M+1.45\,s\mtx I.
\]
The lower bound proves nonsingularity and the remaining assertion in
\cref{eq:shift-R-bounds}. The upper bound is at most
$\bigl(1+1.45s/\sigma_{\min}(\mtx M)^2\bigr)\mtx M^\trans\mtx M$.
Congruence with $\widehat{\mtx R}_1^{-1}$ therefore gives
\cref{eq:shift-cond}, since
$s/\sigma_{\min}(\mtx M)^2=cun\cond_F(\mtx M)^2$.
\end{proof}

\begin{proof}[Proof of Corollary~\ref{cor:scholqr3-stability}]
For $u=0$, the result is exact with $u_{\rm qr3}=0$, its
continuous extension; hence assume $u>0$.
Let $\widehat{\mtx R}_i$ denote the accumulated factor after pass
$i$ of \cref{alg:scholqr3}, and set
$\mtx Q_i=\mtx M\widehat{\mtx R}_i^{-1}$.
By Lemma~\ref{lem:shifted-first-pass} and the definition of $\beta$
in \cref{eq:scholqr3-beta},
\[
\|\mtx Q_1\|_2\le1,
\qquad
\sigma_{\min}(\mtx Q_1)^{-1}\le\sqrt\beta,
\qquad
\cond_F(\mtx Q_1)\le\sqrt{n\beta}.
\]
Since $n\ge10$ and $n\gamma_{n+1}\le1/6$, the same lemma gives
\begin{equation}
\label{eq:qr3-first-preconditioner-bound}
\begin{aligned}
\cond_F(\widehat{\mtx R}_1)
&\le
\sqrt{\frac{1.2+cun^2}{0.55\,cun(1-\gamma_{n+1})}}
\le\frac1{\sqrt n}\sqrt{\frac{2.22}{cu}+1.85n^2},
\\
\cond_F(\mtx M)
&\le\sqrt\beta\,\cond_F(\widehat{\mtx R}_1).
\end{aligned}
\end{equation}
The last inequality uses
$\mtx M=\mtx Q_1\widehat{\mtx R}_1$ and $\|\mtx Q_1\|_2\le1$.

For the second pass, the black-box majorant in
Corollary~\ref{cor:pcholqr-blackbox} satisfies
\begin{align*}
&2.5\,\gamma_n\cond_F(\widehat{\mtx R}_1)
  \bigl(\sqrt n+\cond_2(\mtx Q_1)\bigr)
  +2.81\,Cu\cond_F(\mtx M)
\\
&\quad\le
\frac{\sqrt\beta}{\sqrt n}
\sqrt{\frac{2.22}{cu}+1.85n^2}
\left[2.5(\sqrt n+1)\gamma_n+2.81\,Cu\right]
\le\frac14\sqrt{n\beta}\,u_{\rm qr3}.
\end{align*}
Here and below, comparisons with $u_{\rm qr3}$ follow
coefficientwise from its definition and $n\ge10$.
The first inequality in \cref{eq:scholqr3-smallness} therefore
verifies the first inequality in the sufficient test established in
the proof of Corollary~\ref{cor:pcholqr-blackbox}. The second follows
from
\[
1.6\,n\gamma_{n+1}\cond_F(\mtx Q_1)^2
\le1.6\,n^2\beta\gamma_{n+1}<1,
\]
using \cref{eq:scholqr3-chol-success}.
This test also verifies the inverse-solve condition. Hence the second
pass succeeds, and \cref{thm:pcholqr-stability} yields
\begin{align*}
\|\mtx Q_2^\trans\mtx Q_2-\mtx I\|_F
&\le\frac5{16}\,n\beta u_{\rm qr3}
      +1.75\,n\beta\gamma_{n+1}
\\
&\le\frac1{16}+\frac{1.75}{1.9n}<\frac16.
\end{align*}
Consequently,
\[
\cond_2(\mtx Q_2)\le\sqrt{\frac75}<1.2,
\qquad
\cond_F(\mtx Q_2)\le1.2\sqrt n,
\qquad
\cond_F(\widehat{\mtx R}_2)\le1.2\cond_F(\mtx M).
\]

For the third pass, these estimates and
\cref{eq:qr3-first-preconditioner-bound} give
\begin{align*}
&2.5\,\gamma_n\cond_F(\widehat{\mtx R}_2)
  \bigl(\sqrt n+\cond_2(\mtx Q_2)\bigr)
  +2.81\,Cu\cond_F(\mtx M)
\\
&\quad\le
\frac{\sqrt\beta}{\sqrt n}
\sqrt{\frac{2.22}{cu}+1.85n^2}
\left[3(\sqrt n+1.2)\gamma_n+2.81\,Cu\right]
\le\sqrt{n\beta}\,u_{\rm qr3}.
\end{align*}
Thus the second inequality in \cref{eq:scholqr3-smallness} implies
\[
1.05\,u_{\rm pre}\cond_F(\mtx Q_2)
+1.5\,\gamma_{n+1}\cond_F(\mtx Q_2)^2
\le1.26\,n\sqrt\beta\,u_{\rm qr3}
   +2.16\,n\gamma_{n+1}\le0.2,
\]
where the same sufficient test controls the full coefficient and
the inverse-solve condition.
To verify the remaining Cholesky-success condition, note that
\cref{eq:shifted-pass-smallness} and $\cond_F(\mtx M)^2\ge n$ imply
$\beta\ge2.9\,n^2\gamma_{n+1}$. Together with
\cref{eq:scholqr3-chol-success}, this gives
\[
1.6\,n\gamma_{n+1}\cond_F(\mtx Q_2)^2
\le2.304\,n^2\gamma_{n+1}
\le\frac{2.304}{\sqrt{1.9\cdot2.9}}<1.
\]
Applying \cref{thm:pcholqr-stability} to the third pass now yields
\[
\|\widehat{\mtx Q}^{\trans}\widehat{\mtx Q}-\mtx I\|_F
\le1.5\,n\sqrt\beta\,u_{\rm qr3}
   +2.52\,n\gamma_{n+1}.
\]
All local Cholesky factors and their accumulated products are
nonsingular.
\end{proof}

\subsection{Stability of randomized Q-less CholQR}\label{subsec:stab_rcholqr}
\label{app:proofs_CQRRTO}

\begin{proof}[Proof of Lemma~\ref{lem:computed-sketch-conditioning}]
By the upper embedding inequality,
\[
\|\widehat{\mtx M}^{\rm sk}\|_2
\le\bigl(\sqrt{1+\varepsilon_2}+u_{\rm sk}\bigr)\|\mtx M\|_2.
\]
Hence Assumption~\ref{as:sketch} gives
\begin{align*}
\|\mtx E^{\rm sk}+\mtx E^{\rm qr}\|_2
&\le
\left[(1+u_{\rm qr})u_{\rm sk}
      +u_{\rm qr}\sqrt{1+\varepsilon_2}\right]\|\mtx M\|_2
\\
&=\rho_{\rm sk}\,\sigma_{\min}(\mtx M).
\end{align*}
Using
$\mtx Q^{\rm sk}\mtx P
 =\mtx S\mtx M+\mtx E^{\rm sk}+\mtx E^{\rm qr}$
and the orthonormality of $\mtx Q^{\rm sk}$, we obtain, for every
$\vct x\in\R^n$,
\begin{equation}
\label{eq:computed-sketch-two-sided}
\bigl(\sqrt{1-\varepsilon_1}-\rho_{\rm sk}\bigr)
\|\mtx M\vct x\|_2
\le\|\mtx P\vct x\|_2
\le
\bigl(\sqrt{1+\varepsilon_2}+\rho_{\rm sk}\bigr)
\|\mtx M\vct x\|_2.
\end{equation}
The lower coefficient is positive, so $\mtx P$ is nonsingular.
Substituting $\vct x=\mtx P^{-1}\vct y$ yields
\[
\sigma_{\min}(\mtx M\mtx P^{-1})
\ge\frac1{\sqrt{1+\varepsilon_2}+\rho_{\rm sk}},
\qquad
\sigma_{\max}(\mtx M\mtx P^{-1})
\le\frac1{\sqrt{1-\varepsilon_1}-\rho_{\rm sk}}.
\]
These estimates give the first two bounds in
\cref{eq:CQRRTO-preconditioner-bounds}.
Furthermore, \cref{eq:computed-sketch-two-sided} implies
\[
\|\mtx P\|_F
\le\bigl(\sqrt{1+\varepsilon_2}+\rho_{\rm sk}\bigr)\|\mtx M\|_F,
\qquad
\sigma_{\min}(\mtx P)
\ge\bigl(\sqrt{1-\varepsilon_1}-\rho_{\rm sk}\bigr)
\sigma_{\min}(\mtx M).
\]
Their quotient is the remaining bound
$\cond_F(\mtx P)\le\tau\cond_F(\mtx M)$.
\end{proof}

\begin{proof}[Proof of Corollary~\ref{cor:computed-sketched-preconditioner-stability}]
By Lemma~\ref{lem:computed-sketch-conditioning},
$\mtx P=\widehat{\mtx R}^{\rm sk}$ is nonsingular and
\[
\cond_2(\mtx M^{\rm pre})\le\tau,
\qquad
\cond_F(\mtx M^{\rm pre})\le\sqrt n\,\tau,
\qquad
\cond_F(\mtx P)\le\tau\cond_F(\mtx M).
\]
The black-box majorant in Corollary~\ref{cor:pcholqr-blackbox}
therefore satisfies
\begin{align*}
&2.5\,\gamma_n\cond_F(\mtx P)
 \bigl(\sqrt n+\cond_2(\mtx M^{\rm pre})\bigr)
 +2.81\,Cu\cond_F(\mtx M)
\\
&\qquad\le
 \left[2.5\,\gamma_n\tau(\sqrt n+\tau)+2.81\,Cu\right]
 \cond_F(\mtx M)
 =u_{\rm CQRRTO}\cond_F(\mtx M).
\end{align*}
Together with the conditioning bounds above,
\cref{eq:CQRRTO-smallness} verifies the sufficient test established
in the proof of Corollary~\ref{cor:pcholqr-blackbox}.
Thus $u_{\rm pre}\le u_{\rm CQRRTO}\cond_F(\mtx M)$, both
inequalities in \cref{eq:pcholqr-smallness-full} hold, and
$\gamma_n\cond_F(\mtx P)<1/2$.
Applying \cref{thm:pcholqr-stability} and substituting these bounds
in \cref{eq:pcholqr-Q-bound-full} gives
\cref{eq:computed-sketch-Q-bound}.
The theorem also guarantees that $\widehat{\mtx R}^{\rm chol}$ is
nonsingular, so the returned exact product
$\widehat{\mtx R}^{\rm chol}\mtx P$ is nonsingular.
\end{proof}

\end{document}